\documentclass{article}
\usepackage[%
journal=XXX,    
lang=american,   
]{ems-journal}

\usepackage[noabbrev,capitalize,nameinlink]{cleveref}

\crefformat{equation}{(#2#1#3)}
\crefrangeformat{equation}{(#3#1#4) to~(#5#2#6)}
\crefmultiformat{equation}{(#2#1#3)}{ and~(#2#1#3)}{, (#2#1#3)}{ and~(#2#1#3)}
\crefname{assumption}{Assumption}{Assumptions}

\newcommand{\R}{\mathbb{R}}

\newcommand\abs[1]{\ensuremath{\left.\vert#1\right\vert}}
\newcommand\restr[2]{\ensuremath{\left.#1\right\vert_{#2}}}

\newtheorem{theorem}{Theorem}[section]

\newtheorem{lemma}[theorem]{Lemma}
\newtheorem{corollary}[theorem]{Corollary}
\newtheorem{remark}[theorem]{Remark}
\newtheorem{assumption}[theorem]{Assumption}
\newtheorem{definition}[theorem]{Definition}

\theoremstyle{definition}

\newcommand{\grad}{\nabla}
\newcommand{\norm}[2]{\left\lVert #1\right\rVert_{#2}}

\usepackage{microtype}

\begin{document}

\title{Transmission problems and domain decompositions for non-autonomous parabolic equations on evolving domains}
\titlemark{Transmission problems and domain decompositions for parabolic equations on evolving domains}
\author{Amal Alphonse \and Ana Djurdjevac \and Emil Engström \and Eskil Hansen}

\emsauthor{1}{
	\givenname{Amal}
	\surname{Alphonse}
	\mrid{}
	\orcid{}}{A.~Alphonse}
\emsauthor{2}{
	\givenname{Ana}
	\surname{Djurdjevac}
	\mrid{}
	\orcid{}}{A.~Djurdjevac}
\emsauthor{3}{
	\givenname{Emil}
	\surname{Engström}
	\mrid{}
	\orcid{}}{E.~Engström}
\emsauthor{4}{
	\givenname{Eskil}
	\surname{Hansen}
	\mrid{}
	\orcid{}}{E.~Hansen}
\Emsaffil{1}{
	\department{}
	\organisation{Weierstrass Institute}
	\rorid{00h1x4t21}
	\address{Mohrenstraße 39}
	\zip{10117}
	\city{Berlin}
	\country{Germany}
	\affemail{alphonse@wias-berlin.de}}
\Emsaffil{2}{
	\department{Institut für Mathematik}
	\organisation{Freie Universität Berlin}
	\rorid{1}{046ak2485}
	\address{Arnimallee 6}
	\zip{14195}
	\city{Berlin}
	\country{Germany}
    \affemail{adjurdjevac@zedat.fu-berlin.de}
	\department{2}{} 
	\organisation{2}{}%
	\rorid{2}{}
	\address{2}{}%
	\zip{2}{}
	\city{2}{}
	\country{2}{} 
	\affemail{2}{}}
\Emsaffil{3}{
	\department{Centre for Mathematical Sciences}
	\organisation{Lund University}
	\rorid{012a77v79}
	\address{P.O.\ Box 118}
	\zip{221 00}
	\city{Lund}
	\country{Sweden}
	\affemail{emil.engstrom@math.lth.se}}
\Emsaffil{4}{
	\department{Centre for Mathematical Sciences}
	\organisation{Lund University}
	\rorid{012a77v79}
	\address{P.O.\ Box 118}
	\zip{221 00}
	\city{Lund}
	\country{Sweden}
	\affemail{eskil.hansen@math.lth.se}}

\classification[35K20,65M55]{35R37}

\keywords{evolving domains, moving domains, non-overlapping domain decompositions, transmission problems, parabolic equations, time-dependent Steklov--Poincar\'e operators, convergence, Robin--Robin method}

\begin{abstract}
Parabolic equations on evolving domains model a multitude of applications including various industrial processes such as the molding of heated materials. Such equations are numerically challenging as they require large-scale computations and the usage of parallel hardware. Domain decomposition is a common choice of numerical method for stationary domains, as it gives rise to parallel discretizations. In this study, we introduce a variational framework that extends the use of such methods to evolving domains. In particular, we prove that transmission problems on evolving domains are well posed and equivalent to the corresponding parabolic problems. This in turn implies that the standard non-overlapping domain decompositions, including the Robin--Robin method, become well defined approximations. Furthermore, we prove the convergence of the Robin--Robin method. The framework is based on a generalization of fractional Sobolev--Bochner spaces on evolving domains, time-dependent Steklov--Poincar\'e operators, and elements of the approximation theory for monotone maps. 
\end{abstract}

\maketitle
\section{Introduction}

Industrial applications involving molding typically result in parabolic PDEs with the non-standard feature of evolving or moving spatial domains. In order to illustrate this, consider the production of railway tracks. This process includes two crucial steps, as depicted in~\cref{Fig:rail}. First, the rail is shaped, which involves hot rolling a steel beam with a rectangular cross section into a rail with an H-shaped cross section. Second, the newly molded rail is solidified by spraying water on its surface. A basic model for the temperature $u$ of the rail is then given by a non-autonomous parabolic equation on an evolving domain. That is, for the evolving domain $\{\Omega(t)\}_{t \in [0,\infty)}$, the temperature $u$ satisfies
\begin{equation}\label{eq:strong}
\left\{
     \begin{aligned}
            \dot u(t) - \nabla\cdot\bigl(\alpha(t)\nabla u(t)\bigr) + \bigr(\grad\cdot\mathbf{w}(t) +\beta(t)\bigl)u(t)&=f(t)  & &\text{in }\Omega(t),\\
            u(t)&=\eta(t) & &\text{on }\partial\Omega(t),\\
            u(0)&=0 & &\text{in }\Omega(0).
        \end{aligned}
\right.
\end{equation}
Here, the time evolution is described by the material derivative 
\begin{equation*}
\dot u(t)=\partial_tu(t)+\mathbf{w}(t)\cdot\nabla u(t),
\end{equation*}
and the domain $\Omega(t)$ and its boundary $\partial\Omega(t)$ evolve according to the known velocity field~$\mathbf{w}$. The precise geometry and assumptions on the problem data will be specified in~\cref{sec:movingInterface,sec:weak}. For simplicity we will only consider $\eta\equiv 0$ in \eqref{eq:strong}, but non-zero time-dependent boundary conditions can be handled as done in~\cref{sec:Honehalf}.
\begin{figure}
\begin{center}
    \includegraphics[width=0.45\textwidth]{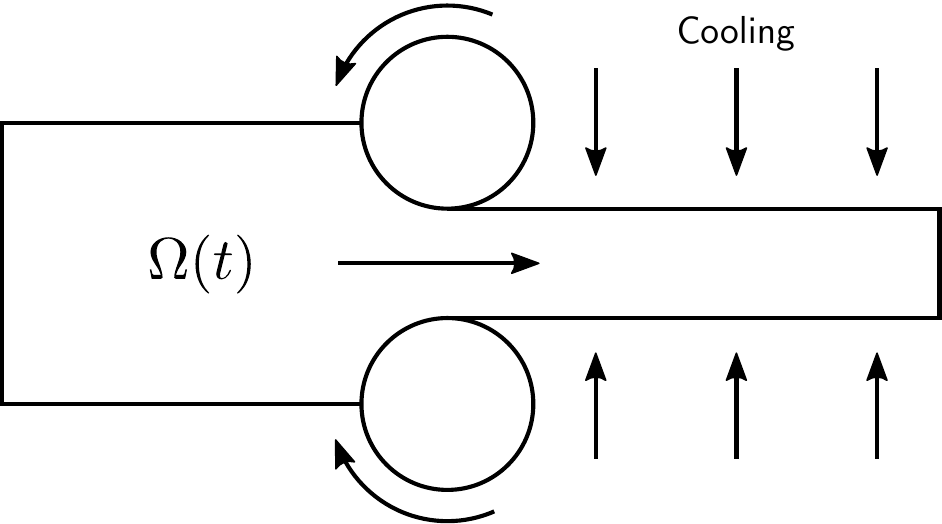}
    \hspace{20pt}\includegraphics[width=0.45\textwidth]{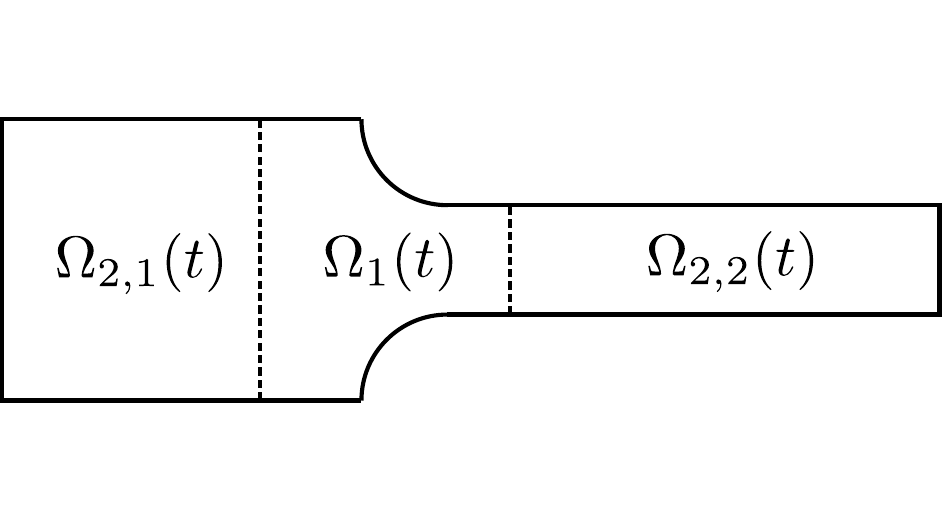}
\end{center}
\caption{\emph{Left image: sketch of a steel beam $\Omega$ at time $t$ being reshaped via hot rolling and thereafter solidified by water cooling. 
Right image: Decomposition of the steel beam at time $t$, where $\Omega(t)$ is decomposed into $\Omega_1(t)$ and $\Omega_2(t)=\Omega_{2,1}(t)\cup \Omega_{2,2}(t)$.}}
\label{Fig:rail}
\end{figure} 

Other applications of parabolic equations on evolving domains and hypersurfaces, to name a few, include: the dynamics of bubbles rising in liquid columns~\cite{HYS09} governed by the Navier--Stokes equations; tumor growth models~\cite{BAR11} consisting of reaction-diffusion equations on surfaces evolving via forced mean curvature flows; spinodal decomposition of binary polymer mixtures~\cite{KRA95} governed by the Cahn--Hilliard equation. 

Parabolic equations on evolving domains are numerically challenging due to the time-dependent geometry and the need for implicit time integration. This all results in large-scale computations that require the usage of parallel and distributed hardware. In the context of stationary domains, the domain decomposition method is a common choice that gives rise to parallel discretizations. The basic idea for these methods is to decompose the domain associated to the equation into subdomains and thereafter communicate the results via the boundaries to the adjacent subdomains. As an example, when shaping the train rail one could, at a fixed time, decompose the rail into three subdomains, where the middle one is a small region around the deformation zone, see~\cref{Fig:rail}. The computational benefit of this is that each subdomain can be given a tailored spatial mesh. For example, the deformation subdomain typically requires a finer mesh than the other subdomains. For a general introduction to domain decomposition methods we refer to~\cite{quarteroni,widlund}.

From a mathematical perspective non-overlapping domain decomposition methods can be designed by first proving that the original parabolic equation~\cref{eq:strong} is equivalent to a so-called \textit{transmission problem}. For two disjoint evolving subdomains $\Omega_i(t)$, $i=1,2$, such that $\overline{\Omega}(t) = \overline{\Omega}_1(t) \cup\overline{\Omega}_2(t)$ and $\Gamma(t) = \partial\Omega_1(t) \cap \partial \Omega_2(t)$, the strong form of the transmission problem becomes
\begin{equation}\label{eq:trans}
    \left\{
	\begin{aligned}
		\dot u_i(t)-\nabla\cdot\bigl(\alpha(t)\nabla u_i(t)\bigr) + \bigr(\grad\cdot\mathbf{w}(t) +\beta(t)\bigl)u_i(t) &=f_i(t)  & &\text{in }\Omega_i(t),\\
		u_i(t)&=0 & &\text{on }\partial\Omega_i(t)\setminus\Gamma(t)\\
		& & & \quad\text{for }i=1,2,\\[5pt]
		u_1(t)&=u_2(t)  & &\text{on }\Gamma(t), \\
		\alpha(t)\nabla u_1(t)\cdot\nu_1(t)+\alpha(t)\nabla u_2(t)\cdot\nu_2(t) &= 0  & &\text{on }\Gamma(t), 
	\end{aligned}
	\right.
\end{equation}
where $\nu_i(t)$ is the unit outward normal vector of $\partial\Omega_i(t)$, $f_i(t)=\restr{f(t)}{\Omega_i(t)}$, and $u_i(t)=\restr{u(t)}{\Omega_i(t)}$. 

The non-overlapping domain decompositions can then be derived by approximating the transmission problem. For example, consider the classic Robin--Robin method, first introduced in~\cite{lions3}. By taking linear combinations of the last two equations in~\cref{eq:trans}, one has the equivalent Robin conditions
\begin{displaymath}
\alpha(t)\nabla u_1(t)\cdot\nu_i(t)+s_0u_1(t)=\alpha(t)\nabla u_2(t)\cdot\nu_i(t)+s_0u_2(t)\quad\text{on }\Gamma(t)\text{ for }i=1,2,
\end{displaymath}
and a method parameter $s_0>0$. Alternating between the subdomains then gives the Robin--Robin method as computing $(u^n_1,u^n_2)$ for $n=1,2,\ldots$ with 
\begin{equation}\label{eq:robin}
\left\{
\begin{aligned}
\dot u^n_1(t)-\nabla \cdot\bigl(\alpha(t)\nabla u^n_1(t)\bigr) +\bigr(\grad\cdot\mathbf{w}(t) +\beta(t)\bigl)u_1^n(t) &=f_1(t)  & &\text{in }\Omega_1(t),\\
u^n_1(t)&=0 & &\text{on }\partial\Omega_1(t)\setminus\Gamma(t),\\
\alpha(t)\nabla u^n_1(t)\cdot\nu_1(t) + s_0 u^n_1(t) &= & & &\\
\alpha(t)\nabla u^{n-1}_2(t)\cdot\nu_1(t) + &s_0 u^{n-1}_2(t)  & &\text{on }\Gamma(t),\\[5pt]
\dot u^n_2(t)-\nabla \cdot\bigl(\alpha(t)\nabla u^n_2(t)\bigr)+\bigr(\grad\cdot\mathbf{w}(t) +\beta(t)\bigl)u_2^n(t)&=f_2(t)  & &\text{in }\Omega_2(t),\\
u^n_2(t)&=0 & &\text{on }\partial\Omega_2(t)\setminus\Gamma(t),\\
\alpha(t)\nabla u^n_2(t)\cdot\nu_2(t) + s_0 u^n_2(t) &= & & &\\
\alpha(t)\nabla u^n_1(t)\cdot\nu_2(t) + &s_0 u^n_1(t) & &\text{on }\Gamma(t).
\end{aligned}
\right.
\end{equation}
Here, $u_2^0$ is an initial guess and $u_i^n(t)$ approximates $u_i(t)=\restr{u(t)}{\Omega_i(t)}$. Note that the Robin--Robin method is sequential, but the computation of each $u_i^n$ can be implemented in parallel when $\Omega_i(t)$ is a union of nonadjacent subdomains, as is the case in~\cref{Fig:rail}.

The well posedness of parabolic equations on evolving domains can be derived via the framework~\cite{amal}, which relies on a
variational formulation where the standard Sobolev--Bochner solution space 
\begin{equation}\label{eq:solspace}
H^1\bigl((0,T);H^{-1}(\Omega)\bigr)\cap L^2\bigl((0,T);H_0^1(\Omega)\bigr)
\end{equation}
is generalized to evolving domains $\Omega(t)$. The framework has also been extended to a Banach space setting~\cite{amal2}. This variational setting constitutes the starting point of the design and analysis of a wide range of finite element methods for equations on evolving domains. The development of continuous-in-time evolving finite element methods have been surveyed in~\cite{ELL21}. The extension to full space-time discretizations via Runge--Kutta and multistep time integrators have, e.g., been analyzed in~\cite{DZI12,KOV18,LUB13}. This type of analysis of full space-time methods has also been extended to parabolic equations given on solution-dependent evolving surfaces~\cite{KOV17}. 

Domain decomposition methods have been proposed in the context of parallel time integrators, as surveyed in~\cite{gander15}, and there are several studies concerning the convergence and other theoretical aspects of space-time decomposition methods applied to parabolic equations on stationary domains, see, e.g.,~\cite{EEEH24,EEEH25,gander07,gander23,japhet20,kwok21,halp12}. However, there is no simple extension of the standard elliptic theory~\cite{quarteroni} to parabolic problems on stationary domains, and certainly not to evolving domains. The main difficulty is that the standard variational setting for parabolic problems, with solutions in the space denoted in~\cref{eq:solspace}, prevents one from deriving the equivalence between~\cref{eq:strong} and~\cref{eq:trans}. This is caused by the fact that functions with the regularity of~\cref{eq:solspace} cannot be ``glued'' together into a new function with the same regularity, see~\cite[Example~2.14]{costabel90}. 

The goals of this paper are therefore to
\begin{enumerate}
\item prove the equivalence between~\cref{eq:strong} and~\cref{eq:trans}, by introducing a suitable variational formulation;
\item demonstrate that the standard non-overlapping domain decomposition methods, including the Robin--Robin method, are well defined on evolving domains;
\item illustrate the applicability of the framework by proving that the Robin--Robin method is convergent when applied to non-autonomous parabolic equations on evolving domains.
\end{enumerate}

The main tool used to achieve the first goal is a new variational formulation with solutions in the evolving domain generalization of the space 
\begin{equation*}
H^{1/2}\bigl((0,\infty);L^2(\Omega)\bigr)\cap L^2\bigl((0,\infty);H_0^1(\Omega)\bigr);
\end{equation*}
see~\cref{sec:abs,sec:comp,sec:weak,sec:Honehalf}. This resolves the issue with ``gluing'' functions together without losing regularity, see~\cref{sec:trans}. The $H^{1/2}$-approach is due to~\cite{lionsmagenes1} for smooth stationary domains and extended to stationary Lipschitz domains in~\cite{costabel90}. We have also explored a $H^{1/2}$-setting for domain decomposition methods for parabolic problems on stationary domains in~\cite{EEEH24,EEEH25}. The study~\cite{Harbrecht} has also used a $H^{1/2}$-variational formulation in order to analyze boundary integral operators for the heat equation on evolving domains. In that work, the standard time derivative is considered instead of the material derivative, which leads to a quite different setting than ours.

The second and third goals are reached by reformulating the transmission problems and the non-overlapping domain decomposition methods in terms of time-dependent Steklov--Poincar\'e operators. As these operators become coercive in the new $H^{1/2}$-setting, one can prove that several domain decomposition methods are well defined via the Lax--Milgram theorem and show convergence for the Robin--Robin method via the abstract convergence result~\cite{lionsmercier}, see~\cref{sec:SP}. Note that this new approach even yields convergence in a stronger norm compared to the previous results on stationary domains~\cite{EEEH24,EEEH25}.  

The continuous analysis, derived in this paper, is also expected to hold in the finite-dimensional case that arises after discretizing in space and time, e.g., by using space-time finite elements, see, e.g.,~\cite{steinbach19}. However, in order to limit the scope of the paper we will restrict ourselves to the continuous case. The features and implementations of full domain decomposition finite element discretizations will be studied elsewhere. 

Throughout the paper, we will use the notation $\mathbb{R}_+=(0,\infty)$, $\mathbb{R}^0_+=[0,\infty)$ and $c$, $C$ will denote generic positive constants.  

\section{Evolving domains}\label{sec:movingInterface}

Let us describe the geometric setting. Assume that we have a bounded Lipschitz domain $\Omega(0) \subset \mathbb{R}^n$, with $n=2,3$, such that
\begin{equation*}
\overline{\Omega(0)} = \overline{\Omega_1(0)} \cup\overline{\Omega_2(0)},
\end{equation*}
where $\Omega_i(0)$, $i=1,2$, are bounded Lipschitz domains that are disjoint with common boundary 
\begin{equation*}
\Gamma(0) = \partial\Omega_1(0) \cap \partial \Omega_2(0).
\end{equation*}
We assume that $\Gamma(0)$ is a $(n-1)$-dimensional Lipschitz manifold. Note that all results will also be valid for the case $n=1$, but with a slightly altered notation. It is also possible to replace $\Omega_i(0)$ by a union of $K_i$ nonadjacent subdomains, i.e.,
\begin{equation*}
    \Omega_i(0)=\bigcup_{k=1}^{K_i}\Omega_{i, k}(0),
\end{equation*}
see~\cref{Fig:rail}, without any change to the analysis.

We now consider $\Omega(0)$ to be evolving in time, resulting in a domain $\Omega(t)$ at a later time~$t$. To this end introduce the velocity field
\begin{equation*}
\mathbf{w}\colon {\R^0_+}\times \mathbb{R}^n \to \mathbb{R}^n
\end{equation*}
together with the corresponding transformation $\Phi\colon {\mathbb{R}^0_+}\times\mathbb{R}^n \to \mathbb{R}^n$ given by 
\begin{equation*}
\frac{\mathrm{d}}{\mathrm{d}t}\Phi_t(x) = \mathbf{w}(t,\Phi_t(x)),\quad t\in \mathbb{R}_+,\quad \Phi_0(x) = x,
\end{equation*}
for all $x \in \mathbb{R}^n$. We will assume that this evolution has the properties described below.
\begin{assumption}\label{ass:phi}
The velocity field $\mathbf{w}$ generates a transformation~$\Phi$ such that 
\begin{enumerate}[label=({\roman*})]
\item $\Phi$ is an element in $C^1({\mathbb{R}^0_+}\times\mathbb{R}^n,\mathbb{R}^n)$ and satisfies the bound
\begin{equation*}
\sup_{t\in\mathbb{R}_+}\|\Phi_t\|_{C^1(\overline{B_r},\mathbb{R}^n)}\leq C<\infty
\end{equation*}
for every fixed ball $B_r=\{x\in\mathbb{R}^n:|x|<r\}$, where $C=C(r)$.
\item The inverse map $\Phi_{-}:(t,x)\mapsto (\Phi_t)^{-1}(x)$ exists and satisfies the same regularity and bound as $\Phi$.
\end{enumerate}
\end{assumption}
The domain at time $t$ is then defined as 
\begin{equation*}
\Omega(t) = \Phi_t(\Omega(0)).
\end{equation*}
The above properties of $\Phi$ imply that $\Omega(t)$ is also a bounded Lipschitz domain with boundary $\partial\Omega(t)=\Phi_t\bigl(\partial\Omega(0)\bigr)$. Furthermore, the Jacobian  $D\Phi_t(x)=\{\partial_{x_j}\Phi_t(x)_i\}_{i,j}$ is well defined and its inverse is given by
\begin{equation*}
(D\Phi_t)^{-1}(x)=(D\Phi_{-t})\bigl(\Phi_t(x)\bigr).
\end{equation*}
We also introduce the determinants
\begin{equation*}
J_t(x)=\det\bigl(D\Phi_t(x)\bigr)\quad\text{and}\quad J_{-t}(x)=\det\bigl(D\Phi_{-t}(x)\bigr)=1/J_t\bigl(\Phi_{-t}(x)\bigr).
\end{equation*}
Regarding the partition of the domain, let us write 
\begin{equation*}
\Omega_i(t)= \Phi_t\bigl(\Omega_i(0)\bigr),\quad i=1,2,
\end{equation*}
for the evolution of the disjoint components. Once again, the assumed properties of $\Phi$ give that all interiors are mapped to interiors and boundaries are mapped to boundaries, thus we obtain the bounded Lipschitz subdomains $\Omega_i(t)$ with the boundaries 
\begin{equation*}
\partial\Omega_i(t) = \Phi_t\bigl(\partial\Omega_i(0)\bigl),\quad i=1,2.
\end{equation*}
We also have
\begin{equation*}
\Phi_t\bigl(\Gamma(0)\bigr) = \Phi_t\bigl(\partial \Omega_1(0) \cap \partial \Omega_2(0)\bigr) = \Phi_t\bigl(\partial \Omega_1(0)\bigr) \cap \Phi_t\bigl(\partial \Omega_2(0)\bigr) = \partial \Omega_1(t) \cap \partial \Omega_2(t),
\end{equation*}
and hence the interface~$\Gamma(0)$ between $\Omega_1(0)$ and $\Omega_2(0)$ is mapped onto the interface between $\Omega_1(t)$ and $\Omega_2(t)$, which we shall call $\Gamma(t)$, i.e.,
\begin{equation*}
\Gamma(t) = \partial \Omega_1(t) \cap \partial \Omega_2(t).
\end{equation*} 
The set $\Gamma(t)$ is again assumed to be an $(n-1)$-dimensional Lipschitz manifold. This setup is exemplified in~\cref{fig:spacetimecylinder}.
\begin{figure}
\centering
\includegraphics[width=0.45\linewidth]{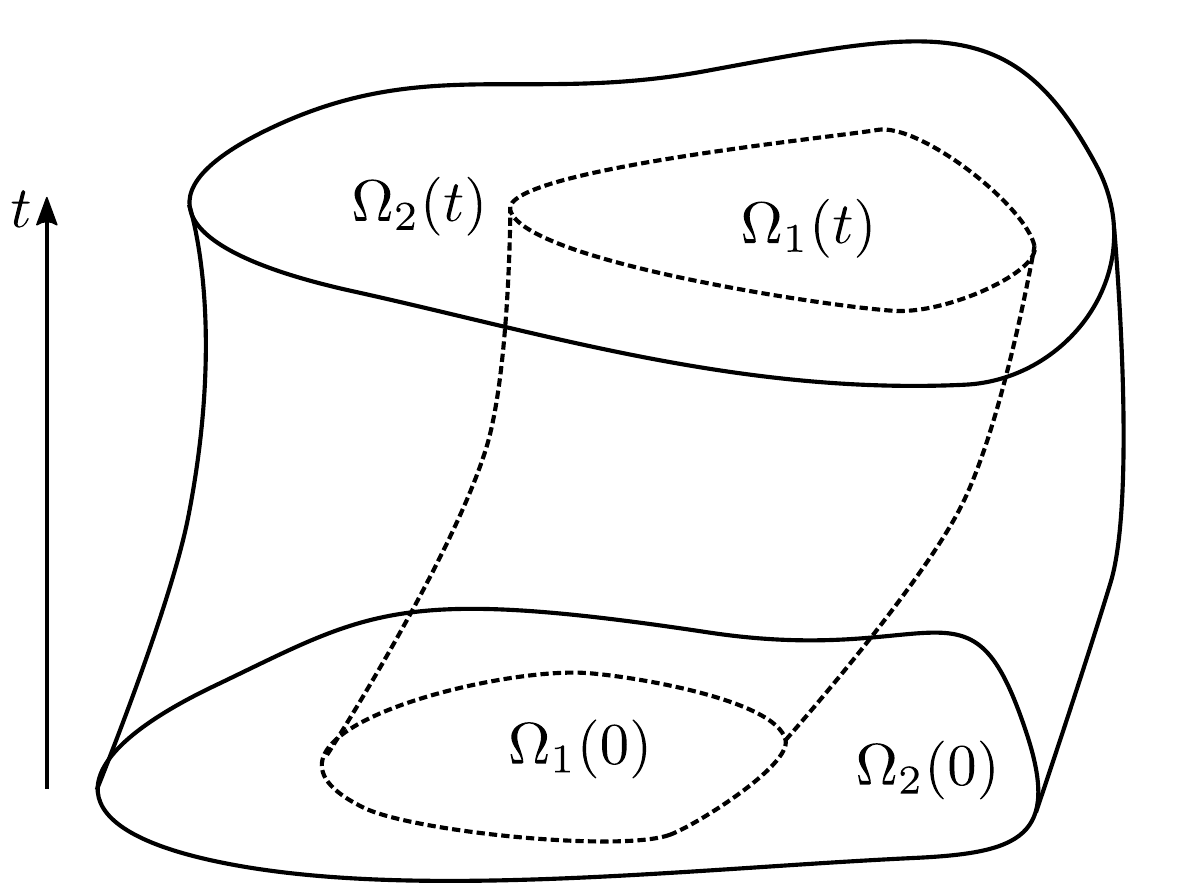}
\caption{An example of an evolving domain decomposition with an interior domain $\Omega_1$ and an exterior domain $\Omega_2$.}
\label{fig:spacetimecylinder}
\end{figure}%
Next, we observe a few more properties of $\Phi$. Let $\Omega(0)\subset B_r$, then~\cref{ass:phi} yields that
\begin{equation*}
\sup_{t \in \mathbb{R}_+}\sup_{x \in B_r} \abs{\Phi_t(x)}\leq C. 
\end{equation*}
Hence, there exists a ball $B_R$ that contains all trajectories of $\Phi$ starting in $\Omega(0)$, i.e., 
$\Omega(t)\subset B_R$ for all $t\in\mathbb{R}^0_+$. We can therefore view $\Phi$ and $\Phi_-$ as maps restricted to $B_r$ and $B_R$, respectively, where 
\begin{equation}\label{eq:Phi_on_B}
\Phi\in C_b\bigl(\mathbb{R}^0_+;C^1(\overline{B_r},\mathbb{R}^n)\bigr)\quad\text{and}\quad\Phi_-\in  C_b\bigl(\mathbb{R}^0_+;C^1(\overline{B_R},\mathbb{R}^n)\bigr).
\end{equation}
Furthermore, there are constants $c,C>0$ such that 
\begin{equation}\label{eq:Phi_biLip}
c|x-y|\leq|\Phi_t(x)-\Phi_t(y)|\leq C|x-y|\quad\text{and}\quad c\leq \abs{J_t(x)}\leq C
\end{equation}
for all $x,y\in B_r$ and $t\in \mathbb{R}^0_+$. These bounds also hold for $\Phi_{-t}$ and $J_{-t}$ with  $x,y\in B_R$. 

\section{Abstract time-evolving function spaces}\label{sec:abs}
In this section we start by generalizing parts of the abstract framework of~\cite{amal,amal2} to a semi-infinite time interval. We then define a notion of a generalized Sobolev--Bochner space that allows in particular for fractional-in-time exponents. This theory will enable us to define the function spaces that we need on the various evolving domains and boundaries from the previous section. We begin with a notion of compatibility, see also \cite[Assumption 2.1]{amal2}.
\begin{definition}\label{def:def_compatibility}
Let $X \equiv \{X(t)\}_{t\in \mathbb{R}^0_+}$ be a family of real separable Hilbert spaces and let
\begin{equation*}
\phi_t\colon X(0) \to X(t)
\end{equation*}
be a linear and invertible map, with its inverse denoted by $\phi_{-t}$. The pair $(X, \phi)$ is said to be \emph{compatible} if
\begin{enumerate}[label=({\roman*})]
\item $\phi_0$ is the identity,
\vspace{-0.2cm}\item there exists a constant $C$ independent of $t \in \R^0_+$ such that
\begin{equation*} 
\begin{aligned}
\norm{\phi_{t} u}{X(t)} &\leq C\norm{u}{X(0)} &&\text{for all }u \in X(0),\\
\norm{\phi_{-t} u}{X(0)} &\leq C\norm{u}{X(t)} &&\text{for all }u \in X(t),
\end{aligned}
\end{equation*}
\vspace{-0.4cm}
\item for all $u \in X(0)$, the map $t \mapsto \norm{\phi_{t} u}{X(t)}$ is measurable.
\end{enumerate}
\end{definition}

If $(X,\phi)$ is compatible, then as done in \cite{amal, amal2}, we may define the $L^2$-space as
\begin{equation*}
        L_X^2(\mathbb{R}_+)=\bigl\{u:\mathbb{R}_+\to\bigcup_{t\in \mathbb{R}_+}X(t)\times\{t\}:\,t\mapsto \bigl(\Bar{u}(t), t\bigr)\text{ such that } \phi_-\Bar{u}\in L^2\bigl(\mathbb{R}_+;X(0)\bigr)\bigr\}.
\end{equation*}
We will abuse notation and simply write $u$ instead of $\Bar{u}$. This is a separable Hilbert space, with the inner product 
\begin{equation*}
(u,v)_{L^2_X(\mathbb{R}_+)} = \int_{\mathbb{R}_+} \bigl(u(t), v(t)\bigr)_{X(t)}\mathrm{d}t;
\end{equation*}
compare with \cite[Theorem 2.8]{amal}. The map 
\begin{equation*}
\phi\colon L^2\bigl(\mathbb{R}_+; X(0)\bigr) \to L^2_X(\mathbb{R}_+)
\end{equation*}
then acts as an isomorphism with an equivalence of norms, see \cite[Lemma 2.10]{amal}. In addition, if $(X,\phi)$ is compatible then so is $(X^*,(\phi_-)^*)$, where $\phi^*:X(t)^*\to X(0)^*$ denotes the dual map of $\phi$. One also has that $L^2_{X}(\mathbb{R}_+)^*\cong L^2_{X^*}(\mathbb{R}_+)$ and 
\begin{equation*}
\begin{aligned}
\langle g, v\rangle_{L^2_{X^*}(\mathbb{R}_+)\times L^2_{X}(\mathbb{R}_+)} &=\int_{\mathbb{R}_+}\langle g(t), v(t)\rangle_{X(t)^*\times X(t)}\mathrm{d}t\\
&=\langle \phi^* g, \phi_- v\rangle_{L^2(\mathbb{R}_+;X(0)^*)\times L^2(\mathbb{R}_+;X(0))}.
\end{aligned}
\end{equation*}
See \cite[p.6 and Lemmas~2.14--15]{amal} for the above statement.

\begin{definition}
For a space $Y\hookrightarrow  L^2(\mathbb{R}_+)$, we use the notation
 \begin{equation*}
 Y_X(\mathbb{R}_+)=\bigl\{v \in  L_X^2(\mathbb{R}_+):\, \phi_-v\in Y\bigl(\mathbb{R}_+;X(0)\bigr)\bigr\}.
\end{equation*}
The norm on $Y_X(\mathbb{R}_+)$ is defined as
\begin{equation*}
\begin{aligned}
        \|v\|_{Y_X(\mathbb{R}_+)}=\|\phi_-v\|_{Y(\mathbb{R}_+; X(0))}.
\end{aligned}
\end{equation*}
\end{definition} 
A consequence of the above is that the restricted map 
\begin{equation*}
\phi\colon Y\bigl(\mathbb{R}_+; X(0)\bigr) \to Y_X(\mathbb{R}_+)
\end{equation*}
is by definition an isometric isomorphism. Note that the above definition of $\|\cdot\|_{Y_X(\mathbb{R}_+)}$ yields an equivalent norm to $\|\cdot\|_{L^2_X(\mathbb{R}_+)}$ for $Y= L^2(\mathbb{R}_+)$, due to compatibility of $(X,\phi)$. However, the norms do not necessarily coincide. 

\begin{lemma}\label{lem:Ysep}
If $(X,\phi)$ is compatible and $Y\bigl(\mathbb{R}_+; X(0)\bigr)$ is a separable Hilbert space, then $Y_X(\mathbb{R}_+)$ is a separable Hilbert space.
\end{lemma}
\begin{proof}
Take a Cauchy sequence $\{y_n\}$ belonging to $Y_X(\mathbb{R}_+).$ By definition, for every $\epsilon>0$, there exists an $N$ such that if $n,m \geq N$, we have
\begin{equation*}
 \norm{\phi_{-}y_n-\phi_{-}y_m}{Y(\mathbb{R}_+;X(0))} = \norm{y_n-y_m}{Y_X(\mathbb{R}_+)} \leq \epsilon.
\end{equation*}
As $Y\bigl(\mathbb{R}_+;X(0)\bigr)$ is a Hilbert space, $\phi_{-}y_n \to z$ in $Y\bigl(\mathbb{R}_+;X(0)\bigr)$ to some $z\in Y\bigl(\mathbb{R}_+;X(0)\bigr)$. Since $Y\bigl(\mathbb{R}_+;X(0)\bigr)$ is a subset of $L^2\bigl(\mathbb{R}_+;X(0)\bigr)$ and we have compatibility, it follows that $\phi z \in L^2_X.$ Hence, 
$\phi z \in Y_X(\mathbb{R}_+).$ We have
\begin{equation*}
\norm{y_n - \phi z}{Y_X(\mathbb{R}_+)} = \norm{\phi_-y_n - z}{Y(\mathbb{R}_+; X(0))} \to 0
\end{equation*}
as $n$ tends to infinity. Thus, the sequence $\{y_n\}$ is convergent in $Y_X(\R_+)$. This shows that every Cauchy sequence converges, hence it is complete and a Hilbert space.

For the separability, if $\{e_i\}$ is a countable orthonormal basis of $Y\bigl(\mathbb{R}_+; X(0)\bigr)$, we can write an arbitrary $z \in Y\bigl(\mathbb{R}_+; X(0)\bigr)$ as $z = \sum (z, e_i)_{Y(\mathbb{R}_+; X(0))}e_i.$ Then we have
\begin{equation*}
 \textstyle{       \phi z = \sum_i (z, e_i)_{Y(\mathbb{R}_+; X(0))}\phi e_i = \sum_i (\phi z, \phi e_i)_{Y_X(\mathbb{R}_+)}\phi e_i.}
\end{equation*}
Since any element of $Y_X(\mathbb{R}_+)$ can be written as $\phi z$ for some $z \in Y\bigl(\mathbb{R}_+; X(0)\bigr)$ and because $L^2_X$ and $L^2\bigl(\mathbb{R}_+; X(0)\bigr)$ are isomorphic via $\phi$, we see that $\{\phi e_i\}$ is a countable dense subset and thus the space is separable.
\end{proof}

Consider a space $Y\hookrightarrow  L^2(\mathbb{R}_+)$ and assume that $Y\bigl(\mathbb{R}_+; X(0)\bigr)$ is a separable Hilbert space. Next, introduce two compatible pairs $(X,\phi)$ and $(Z,\phi)$, where $Z(t)\hookrightarrow  X(t)$ for all $t\in\mathbb{R}_+$. The subset 
\begin{equation*}
Y\bigl(\mathbb{R}_+; X(0)\bigr)\cap L^2\bigl(\mathbb{R}_+; Z(0)\bigr)\subseteq L^2\bigl(\mathbb{R}_+; X(0)\bigr)
\end{equation*}
then also becomes a separable Hilbert space when equipped with the inner product
\begin{equation*}
u,v\mapsto(u,v)_{Y(\mathbb{R}_+; X(0))}+(u,v)_{ L^2(\mathbb{R}_+; Z(0))}.
\end{equation*}
The same holds for $Y_{X}(\mathbb{R}_+)\cap L^2_Z(\mathbb{R}_+)\subseteq L^2_X(\mathbb{R}_+)$ by~\cref{lem:Ysep}. The restriction of the map $\phi\colon L^2\bigl(\mathbb{R}_+; X(0)\bigr) \to L^2_X(\mathbb{R}_+)$, i.e.,
\begin{equation*}
\phi\colon Y\bigl(\mathbb{R}_+; X(0)\bigr)\cap L^2\bigl(\mathbb{R}_+; Z(0)\bigr)\to Y_{X}(\mathbb{R}_+)\cap L^2_Z(\mathbb{R}_+),
\end{equation*}
is an isometric isomorphism.

\section{Compatibility for spaces defined on evolving domains}\label{sec:comp}

In this section we wish to apply the theory of \cref{sec:abs} to concrete function spaces on evolving domains and boundaries. This mainly involves checking that compatibility holds in the sense of \cref{def:def_compatibility}. The results here improve those present in \cite{amal, amal2, AESIFB} because we assume much weaker regularity on the domains and the evolution than there.

Let $M \equiv \{M(t)\}_{t\in \mathbb{R}^0_+}$ be a family of Lipschitz domains, with $M(0)\subset B_r\subset\mathbb{R}^n$ and $M(t)\subset B_R\subset\mathbb{R}^n$, playing the role of $\Omega_i$ or $\Omega$. Similarly, let  $S$ be a family of $(n-1)$-dimensional Lipschitz manifolds 
representing either~$\partial\Omega_i$, $\partial\Omega$, or $\Gamma$. We refer to~\cite[Sections~6.2--3]{kufner} and~\cite{Naumann2011Measure} for an in-depth treatment of Lipschitz manifolds and the related surface integrals. 

In the setting of evolving domains, we will consider two families of compatible pairs given by the maps $\phi$ and $\psi$ relating to function spaces over $M$ and $S$, respectively.  
\begin{definition}
The maps $\phi$ and $\psi$, together with their inverses, are identified as the compositions
\begin{enumerate}[label=({\roman*})]
\item $\phi_tu=u\circ \Phi_{-t}\quad\text{and}\quad\phi_{-t}u=u\circ \Phi_t$,
\item $\psi_tu=u\circ \bigl(\Phi_{-t}|_{S(t)}\bigr)\quad\text{and}\quad\psi_{-t}u=u\circ \bigl(\Phi_{t}|_{S(0)}\bigr)$,
\end{enumerate}
respectively.
\end{definition}
Note that this choice of $\phi$ and $\psi$ trivially fulfills the first compatibility property of~\cref{def:def_compatibility}. For notational simplicity, we also make use of the notation $\phi,\psi$ for matrix-valued maps, e.g., for $A:M(0)\to \mathbb{R}^{n\times n}$ we can write 
\begin{equation*}
\phi_{-t}A(x)=(A\circ\Phi_{t})(x)=A\bigl(\Phi_{t}(x)\bigr).
\end{equation*}

The task is now to prove the remaining compatibility properties for the function spaces arising when deriving the weak formulation of~\cref{eq:strong}. 

\subsection{Spaces on the interior}

\begin{lemma}\label{lem:H1Compatibility}
If \cref{ass:phi} holds then $(L^2(M), \phi)$ and $(H^1(M), \phi)$ are compatible pairs. 
\end{lemma}

\begin{proof}
We prove the $H^1$ case as the $L^2$ case follows by a simpler argument. To prove the second property of~\cref{def:def_compatibility}, consider $u\in H^1\bigl(M(0)\bigr)$ and $v=\phi_t u\in H^1\bigl(M(t)\bigr)$. One then has, pointwise a.e., the formulae
\begin{equation*}
\begin{gathered}
\nabla(\phi_{-t} v) = (D\Phi_t)^T \phi_{-t}(\nabla v)\quad\Leftrightarrow\quad
\nabla u\ = (D \Phi_t)^T \phi_{-t}(\nabla \phi_{t} u)\quad\Leftrightarrow\\
\phi_{-t}(\nabla \phi_{t} u)= (D \Phi_t)^{-\mathrm{T}}\nabla u\quad\Leftrightarrow\quad
\phi_{-t}(\nabla \phi_{t} u)=(\phi_{-t}D\Phi_{-t})^{\mathrm{T}}\nabla u.
\end{gathered}
\end{equation*}
This yields the identities 
\begin{align}
 \nonumber \norm{\phi_t u}{H^1(M(t))}^2 &= \int_{M(t)} |\phi_t u|^2+|\nabla \phi_t u|^2\,\mathrm{d}x_t\\
 \nonumber &=\int_{M(0)}  \phi_{-t}\bigl(|\phi_t u|^2+|\nabla \phi_t u|^2\bigr)|J_t|\,\mathrm{d}x_0\\
&=\int_{M(0)}  u^2|J_t|+|(\phi_{-t}D\Phi_{-t})^{\mathrm{T}}\grad u|^2|J_t| \,\mathrm{d}x_0.\label{eq:H1comp}
\end{align}
By~\cref{ass:phi} we have $D\Phi_-\in C({\mathbb{R}^0_+}\times\mathbb{R}^n,\mathbb{R}^{n\times n})$, $|J|\in C({\mathbb{R}^0_+}\times\mathbb{R}^n,\mathbb{R})$, and the bound 
\begin{equation}\label{eq:Phibound}
\begin{aligned}
 &\sup_{t\in\mathbb{R}_+}  \sup_{x\in M(0)} |(\phi_{-t}D\Phi_{-t})^{\mathrm{T}}(x)|_{2}^2\,|J_t(x)|\\
&\quad\leq  \sup_{t\in\mathbb{R}_+}  \sup_{x\in M(0),y\in M(t)}|(D\Phi_{-t})(y)|_F^2|J_t(x)|\\
&\quad\leq \sup_{t\in\mathbb{R}_+} \sup_{x\in B_r,y\in B_R} \max_{i,j=1,\ldots,n}6n^2|\partial_{y_j}\Phi_{-t}(y)_i|^2\,|\partial_{x_j} \Phi_t(x)_i|^n\\
&\quad\leq  C(n)\bigl(\sup_{t\in\mathbb{R}_+}\|\Phi_{-t}\|_{C^1(\overline{B_R},\mathbb{R}^n)}\bigr)^2\bigl(\sup_{t\in\mathbb{R}_+}\|\Phi_{t}\|_{C^1(\overline{B_r},\mathbb{R}^n)}\bigr)^n<\infty,
\end{aligned}
\end{equation}
where $|\cdot|_{2}$ and $|\cdot|_{F}$ refer to the Euclidean and Frobenius matrix norms respectively. Hence, the second term of~\cref{eq:H1comp} is bounded by $C\|\nabla u\|^2_{L^2(M(0),\mathbb{R}^n)}$ and a similar argument for the first term yields that 
\begin{equation*}
\norm{\phi_{t} u}{H^1(M(t))} \leq C\norm{u}{H^1(M(0))}\quad\text{for all }u \in H^1\bigl(M(0)\bigr),
\end{equation*}
with a constant $C$ independent of $t \in \R^0_+$. The reverse bound, i.e., $\norm{\phi_{-t} u}{H^1(M(0))} \leq C\norm{u}{H^1(M(t))}$, holds as the very same properties are assumed for $\Phi$ and $\Phi_-$.

We finally consider the third property of~\cref{def:def_compatibility}. Again by~\cref{ass:phi}, we have that $|J| \in L^\infty\bigl(\mathbb{R}_+;L^\infty(M(0))\bigr)$ and hence $|J|\colon \mathbb{R}_+\to L^\infty\bigl(M(0)\bigr)$ is strongly measurable. By Pettis' theorem, it is also weakly measurable. For an element $v\in L^\infty\bigl(M(0)\bigr)$ the functional $g(v) = \int_{M(0)}u^2v\,\mathrm{d}x_0$ clearly satisfies $g \in L^\infty\bigl(M(0)\bigr)^*$ since $u \in L^2\bigl(M(0)\bigr)$. Thus
\begin{equation*}
t \mapsto g(|J_t|) = \int_{M(0)}u^2 |J_t|\,\mathrm{d}x_0
\end{equation*}
is measurable.  This yields measurability of the first term of the integral~\cref{eq:H1comp}. A similar argument gives measurability of the entire integral.
\end{proof}
\subsection{Spaces on the boundary}
Let us now address function spaces defined over boundaries.
\begin{lemma}\label{lem:compatL2S}
If \cref{ass:phi} holds then $\bigl(L^2(S), \psi\bigr)$ is a compatible pair. 
\end{lemma}

\begin{proof}
For simplicity, we first consider the case $n=2$. The curve $S(0)$ is then the union of finitely many open, overlapping sets 
\begin{equation*}
S_\ell(0)=\{x\in \mathbb{R}^2: x(\xi)=\tilde{A}_\ell^{-1}\bigr(\xi,\sigma_\ell(\xi)\bigl)^\mathrm{T}\text{ for }\xi\in[-a,a]\}.
\end{equation*}
Here, $\sigma_\ell\colon [-a,a]\to\mathbb{R}$ are Lipschitz continuous maps, and $\tilde{A}_\ell$ are affine transformations, i.e., $\tilde{A}_\ell x=A_\ell x+b_\ell$, where $A_\ell$ are orthonormal matrices with $\mathrm{det}A_\ell=1$. The curve integral over $S(0)$ can then be defined as 
\begin{equation*}
\int_{S(0)} v \,\mathrm{d}s_0 = \sum_{\ell=1}^M \int_{-a}^a (\varphi_\ell v)\bigl(x(\xi)\bigr)\,|x_\xi(\xi)|\,\mathrm{d}\xi, 
\end{equation*}
with the tangential derivative $x_\xi=A_\ell^{-1}(1,\sigma'_\ell)^\mathrm{T}\in L^\infty\bigl(S_\ell(0),\mathbb{R}^2\bigr)$, for any partition of unity $\{\varphi_\ell\}\subset C(S(0))$ of the curve~$S(0)$. Note that the regularity of $x_\xi$ follows by the Lipschitz continuity of $\sigma_\ell$, see~\cite[Theorem~6.2.14]{kufner}. Also observe that the denominator $|x_\xi|^2=1+(\sigma'_\ell)^2$ is nonzero.

Let $y(\xi)=\Phi_t x(\xi)$, then $y_\xi=\mathrm{D}\Phi_t(x)x_\xi$. Furthermore, if $\{\varphi_\ell\}\subset C\bigl(S(t)\bigr)$ is a partition of unity of the curve~$S(t)$ then $\{\psi_{-t}\varphi_\ell\}\subset C\bigl(S(0)\bigr)$ becomes a partition of unity of~$S(0)$. We then have 
\begin{equation}\label{eq:intident}
\begin{aligned}
\int_{S(t)} v(y) \,\mathrm{d}s_t &= \sum_{\ell=1}^M \int_{-a}^a (\varphi_\ell v)\bigl(y(\xi)\bigr)\,|y_\xi(\xi)|\,\mathrm{d}\xi\\
                                         &= \sum_{\ell=1}^M \int_{-a}^a (\varphi_\ell v)\bigl(\Phi_t x(\xi))\,\left|\mathrm{D}\Phi_t\bigr(x(\xi)\bigr)\frac{x_\xi(\xi)}{|x_\xi(\xi)|}\right|\,
                                              |x_\xi(\xi)|\mathrm{d}\xi\\
                                         &= \int_{S(0)} (\psi_{-t}v)(x)\,\omega_{2,t}(x)\,\mathrm{d}s_0,
\end{aligned}
\end{equation}
where $\omega_{2,t}(x)=|\mathrm{D}\Phi_t(x)\tau(x)|$ with $\tau\in L^\infty\bigl(S(0),\mathbb{R}^2\bigr)$ denoting the normalized tangent vector of $S(0)$. By~\cref{ass:phi} we obtain that $\omega_{2}\in L^\infty\bigl(\mathbb{R}_+;L^\infty\bigl(S(0))\bigr)$ as
\begin{equation}\label{eq:omegabound}
\begin{aligned}
&\omega_{2,t}(x) \leq |\mathrm{D}\Phi_t(x)|_2 |\tau(x)|_2\leq \sup_{y\in S(0)} |\mathrm{D}\Phi_t(y)|_F\cdot 1\\
&\qquad\leq \sup_{t\in\mathbb{R}_+}\sup_{y\in B_r}\max_{i,j=1,\ldots,n} n|\partial_{y_j}\Phi_{t}(y)_i|\leq C\sup_{t\in\mathbb{R}_+}\|\Phi_{t}\|_{C^1(\overline{B_r},\mathbb{R}^n)}<\infty
\end{aligned}
\end{equation}
for a.e.\ $x\in S(0)$ and every $t\in\mathbb{R}_+$.

From~\cref{eq:intident} it is clear that $\psi_{-t}v\in L^1\bigl(S(0)\bigr)$ if $v\in L^1\bigl(S(t)\bigr)$ and $\psi_{t}v\in L^1\bigl(S(t)\bigr)$ if $v\in L^1\bigl(S(0)\bigr)$. Next, consider $u\in L^2\bigl(S(0)\bigr)$. Replacing $v$ by 
$\psi_{t} u^2$ in~\cref{eq:intident} yields that
\begin{equation*}
\|\psi_{t} u\|_{L^2(S(t))}=\| u\,\sqrt{\omega_{2,t}}\|_{L^2(S(0))}\leq C \|u\|_{L^2(S(0))}.
\end{equation*}
Here, the constant~$C$ is uniform in time by~\cref{eq:omegabound}. The reverse bound of~\cref{def:def_compatibility} follows by simply replacing $S(t)$ with $S(0)$ and vice versa in the above argumentation. 

The measurability, i.e., the third property of~\cref{def:def_compatibility}, can be shown in a similar way to \cref{lem:H1Compatibility}. Hence, $\bigl(L^2(S), \psi\bigr)$ is a compatible pair for $n=2$.

The compatibility for $n=3$ follows in the same fashion with 
\begin{equation*}
S_\ell(0)=\{x\in \mathbb{R}^3: x(\xi)=\tilde{A}_\ell^{-1}\bigr(\xi_1,\xi_2,\sigma_\ell(\xi)\bigl)^\mathrm{T}\text{ for }\xi\in[-a,a]^2\},
\end{equation*}
$x_\xi$ replaced by $\partial_{\xi_1} x\,\times\,\partial_{\xi_2} x$, and $\omega_{2,t}$ replaced by 
\begin{equation*}
\omega_{3,t}\bigr(x(\xi)\bigr)=\frac{|\mathrm{D}\Phi_t\bigl(x(\xi)\bigr)\partial_{\xi_1} x(\xi)\,\times\,\mathrm{D}\Phi_t\bigl(x(\xi)\bigr)\partial_{\xi_2} x(\xi)|}{|\partial_{\xi_1} x(\xi)\,\times\,\partial_{\xi_2} x(\xi)|}.
\end{equation*}
Note that the replaced terms are all in~$L^\infty$, $|\partial_{\xi_1} x\,\times\,\partial_{\xi_2} x|^2=1+\sum_{i=1}^2(\partial_{\xi_i}\sigma_\ell)^2$ is nonzero, and the bound~\cref{eq:omegabound} holds as $Ax\times Ay=\det(A)A^{-\mathrm{T}}x\times y$.
\end{proof}
Regarding Sobolev spaces over the manifolds $S$, we introduce the space $H^{1/2}\bigl(S(t)\bigr)$ defined as
\begin{equation*}\label{eq:slobodetskii}
	\begin{gathered}
	    H^{1/2}\bigl(S(t)\bigr)=\{u\in L^2\bigl(S(t)\bigr):  \|u\|_{H^{1/2}(S(t))}<\infty\}\quad\text{with}\\
	    \|u\|_{H^{1/2}(S(t))} =
	                 \Big(|u|_{H^{1/2}(S(t))}^2
	                        +\|u\|^2_{L^2(S(t))}\Big)^{1/2}\quad\text{and}\\
	    |u|_{H^{1/2}(S(t))} =
	                 \Big(\int_{S(t)}\int_{S(t)}\frac{\abs{u(x)-u(y)}^2}{\abs{x-y}^{d}}\,\mathrm{d}s_t\,\mathrm{d}s_t\Big)^{1/2}.
	\end{gathered}
\end{equation*}
Denoting the extension by zero from $\Gamma(t)$ to $\partial\Omega_i(t)$ by $e_{\partial\Omega_i(t)}$, we also define the Lions--Magenes space as
\begin{equation*}
	\begin{gathered}
	    \Lambda(t)=\{u\in L^2\bigl(\Gamma(t)\bigr): e_{\partial\Omega_i(t)}u\in H^{1/2}\bigl(\partial\Omega_i(t)\bigr)\}\quad\text{with}\\ 
	    \|u\|_{\Lambda(t)}=\|e_{\partial\Omega_i(t)}u\|_{H^{1/2}(\partial\Omega_i(t))}.
\end{gathered}
\end{equation*}
By~\cite[Lemma A.8]{widlund} one has the identification $\Lambda(t)\cong[H^{1/2}_0\bigl(\Gamma(t)\bigr), L^2\bigl(\Gamma(t)\bigr)]_{1/2}$, i.e., $\Lambda(t)$ is independent of $i=1,2$. The space $H^{1/2}\bigl(S(t)\bigr)$ is a separable Hilbert space, and the same therefore holds for $\Lambda(t)$. 

\begin{lemma}\label{lem:compatH1/2}
If \cref{ass:phi} holds then $\bigl(H^{1/2}(S), \psi\bigr)$ is a compatible pair. 
\end{lemma}

\begin{proof}
The proof is not dissimilar to the discussion in \cite[\S 5.4.1]{AESIFB}. With the same notation as in~\cref{lem:compatL2S} we have
\begin{align}
     \nonumber 		|\psi_t u|_{H^{1/2}(S(t))}^2 &= \int_{S(t)}\int_{S(t)} \frac{|\psi_t u(y)-\psi_t u(\hat y)|^2}{|y-\hat y|^n}\,\mathrm{d}s_t\mathrm{d}s_t\\
     \nonumber 		             &= \int_{S(t)}\int_{S(t)} \frac{|u(\Phi_{-t} (y))- u(\Phi_{-t}(\hat y))|^2}{|y-\hat y|^n}\,\mathrm{d}s_t\mathrm{d}s_t\\
          	 &= \int_{S(0)}\int_{S(0)} \frac{|u(x)- u(\hat x)|^2}{|\Phi_t (x)- \Phi_t (\hat x)|^n}\omega_{n,t}(x)\omega_{n,t}(\hat x)\,\mathrm{d}s_0\mathrm{d}s_0\label{eq:to_refer_1}\\
     \nonumber 			     &\leq C\int_{S(0)}\int_{S(0)} \frac{|u(x)- u(\hat x)|^2}{|x- \hat x|^n}\omega_{n,t}(x)\omega_{n,t}(\hat x)\,\mathrm{d}s_0\mathrm{d}s_0\\
     \nonumber 			     &\leq C\int_{S(0)}\int_{S(0)} \frac{|u(x)- u(\hat x)|^2}{|x- \hat x|^n}\,\mathrm{d}s_0\mathrm{d}s_0,
\end{align}
where the constant~$C$ is uniform in time. The bounds follow as $\Phi_t$ and $\omega_{n}$ fulfill~\cref{eq:Phi_biLip,eq:omegabound}, respectively. Since the bound for $\|\psi_t u\|_{L^2(S(t))}$ follows as in~\cref{lem:compatL2S} we have $\psi_t \colon H^{1/2}\bigl(S(0)\bigr) \to H^{1/2}\bigl(S(t)\bigr)$ and 
\begin{equation*}
\|\psi_{t} u\|_{H^{1/2}(S(t))}\leq C \|u\|_{H^{1/2}(S(0))}.
\end{equation*}
A similar argument gives the same result for $\psi_{-t}$. 

Finally, we prove the third property of~\cref{def:def_compatibility}. By \cref{ass:phi} we have that $\Phi\colon\mathbb{R}^0_+ \to C\bigl(S(0)\bigr)$ and $\omega_{n} \colon \mathbb{R}_+\to L^\infty\bigl(S(0)\bigr)$ are continuous. Hence, the integrand in \eqref{eq:to_refer_1} is continuous in time for all fixed $x, \hat x$ excluding the zeros set where $\omega_{n,t}$ is undefined and $x = \hat x$. The set violating the latter condition is again of measure zero. The said integrand is bounded from above by 
\begin{equation*}
C\frac{|u(x)- u(\hat x)|^2}{|x- \hat x|^n}
\end{equation*}
via~\cref{eq:Phi_biLip,eq:omegabound}, and this function is integrable and independent of time. We may therefore apply Lebesgue's dominated convergence theorem to deduce that the right-hand side of \eqref{eq:to_refer_1} (and hence also $\|\psi_t u\|_{H^{1/2}(S(t))}$) is continuous with respect to $t$, thereby yielding the sought-after measurability.
\end{proof}

\begin{lemma}\label{lem:extensionOpCommutesWithDiffeo}
If \cref{ass:phi} holds then 
\begin{equation*}
e_{\partial\Omega_i(t)} \psi_{t} u = \psi_{t} e_{\partial\Omega_i(0)}u
\end{equation*}
for all $u \in L^2\bigl(\Gamma(0)\bigr)$.
\end{lemma}
\begin{proof}
By~\cref{lem:compatL2S}, one obtains that $e_{\partial\Omega_i(t)} \psi_{t} u$ and $\psi_{t} e_{\partial\Omega_i(0)}u$ are in $L^2\bigl(\partial\Omega_i(t)\bigr)$. 
We therefore have, for a.e.\ $x \in \partial \Omega_i(t)$,
    \begin{align*}
        (e_{\partial\Omega_i(t)} \psi_{t} u)(x) = \begin{cases}
            (\psi_{t}  u)(x) & \text{if $x \in \Gamma(t)$}\\
            0 & \text{if $x \in \partial\Omega_i(t)\backslash\Gamma(t)$},
        \end{cases}
    \end{align*}
    while on the other hand
    \begin{align*}
        (\psi_{t} e_{\partial\Omega_i(0)}u)(x) = (e_{\partial\Omega_i(0)}u)\bigl(\Phi_{-t}(x) \bigr)&= \begin{cases}
            u(\Phi_{-t}(x)) & \text{if $\Phi_{-t}(x) \in \Gamma(0)$}\\
            0 & \text{if $\Phi_{-t}(x) \in \partial\Omega_i(0)\backslash \Gamma(0)$}
        \end{cases}\\
        &= \begin{cases}
            (\psi_{t} u)(x) & \text{if $x \in \Gamma(t)$}\\
            0 & \text{if $x \in \partial\Omega_i(t)\backslash \Gamma(t)$}.
        \end{cases}
    \end{align*}
Here we used that $\Phi_t\bigl(\partial \Omega_i(0)\bigr)=\partial\Omega_i(t)$ and $\Phi_t\bigl(\Gamma(0)\bigr)=\Gamma(t)$.
\end{proof}

\begin{lemma}\label{lem:compatLambda}
If \cref{ass:phi} holds then $(\Lambda, \psi)$ is a compatible pair. 
\end{lemma}

\begin{proof}
Let $u\in\Lambda(0)$ then $e_{\partial\Omega_i(0)}u\in H^{1/2}(\partial\Omega_i(0))$. By~\cref{lem:compatH1/2,lem:extensionOpCommutesWithDiffeo}, one obtains that 
\begin{equation*}
\psi_{t} e_{\partial\Omega_i(0)} u=e_{\partial\Omega_i(t)} \psi_{t} u\in H^{1/2}(\partial\Omega_i(t)).
\end{equation*}
From the definition of $\Lambda(t)$ (cf.~\cite[Lemma~4.1]{EHEE22}), there is a unique element $v\in\Lambda(t)$ such that $e_{\partial\Omega_i(t)}v=e_{\partial\Omega_i(t)} \psi_{t} u$, i.e., $\psi_{t} u=v$. Hence, $\psi_t$ maps~$\Lambda(0)$ into~$\Lambda(t)$.
By the definition of $\norm{\cdot}{\Lambda(t)}$ together~\cref{lem:compatH1/2,lem:extensionOpCommutesWithDiffeo}, we have 
\begin{equation*}
\begin{aligned}
\norm{\psi_t u}{\Lambda(t)} &=\norm{e_{\partial\Omega_i(t)}\psi_t u}{H^{1/2}(\partial\Omega_i(t))}
                                              =\norm{\psi_te_{\partial\Omega_i(0)} u}{H^{1/2}(\partial\Omega_i(t))}\\
                                            &\leq C\norm{e_{\partial\Omega_i(0)} u}{H^{1/2}(\partial\Omega_i(0))}= C\norm{u}{\Lambda(0)}.
\end{aligned}
\end{equation*}
The same line of reasoning can be made for $\psi_-$ and the measurability of $t\mapsto\norm{\psi_t u}{\Lambda(t)}$ follows just as in~\cref{lem:compatH1/2}. Hence, $(\Lambda, \psi)$ is a compatible pair. 
\end{proof}

To tie together the functions over $M$ and $S$ we consider the linear, bounded, and surjective trace operator 
$T_{\partial\Omega_i(t)}: H^1\bigl(\Omega_i(t)\bigr)\rightarrow H^{1/2}\bigl(\partial\Omega_i(t)\bigr)$, see~\cite[Theorem 6.8.13]{kufner}, together with the space
\begin{equation*}
V_i(t)=\{u\in H^1\bigl(\Omega_i(t)\bigr): \restr{(T_{\partial\Omega_i(t)}u)}{\partial\Omega_i(t)\backslash\Gamma(t)}=0\}.
\end{equation*}
The spaces $V_i(t)$ and $H^1_0\bigl(\Omega_i(t)\bigr)$ are both equipped with the norm $\|\cdot\|_{H^1(\Omega_i(t))}$ and are separable Hilbert spaces. For future reference, we also introduce the trace operator on $V_i(t)$ by
\begin{equation*}
	     T_{i,t}\colon V_i(t)\rightarrow \Lambda(t), \qquad u\mapsto\restr{(T_{\partial\Omega_i(t)}u)}{\Gamma(t)},
\end{equation*}	     
which is again linear, bounded, and surjective, see~\cite[Lemma 4.4]{EHEE22}.
\begin{lemma}\label{lem:commutationOfphiAndTrace}
If \cref{ass:phi} holds then
\begin{equation*}
T_{\partial\Omega_i(t)}\phi_t u  = \psi_tT_{\partial\Omega_i(0)} u\quad\text{and}\quad T_{i,t}\phi_t v  = \psi_tT_{i,0} v
\end{equation*}
for all $u \in H^1\bigl(\Omega_i(0)\bigr)$ and $v \in V_i(0)$.
\end{lemma}
\begin{proof}
For $\varphi \in C^\infty(\overline{\Omega_i(0)})$, we have
\begin{equation*}
T_{\partial\Omega_i(t)}(\phi_t\varphi)  = T_{\partial\Omega_i(t)}\bigl(\varphi\circ\Phi_{-t})  = (\varphi \circ \Phi_{-t})|_{\partial\Omega_i(t)},
\end{equation*}
with the second equality because $\varphi \circ \Phi_{-t} \in C(\overline{\Omega_i(0)})$ due to~\cref{eq:Phi_on_B}. We also have
\begin{equation*}
\psi_t(T_{\partial\Omega_i(0)}\varphi) = \psi_t(\varphi|_{\partial\Omega_i(0)}) = \varphi|_{\partial\Omega_i(0)} \circ (\Phi_{-t}|_{\partial\Omega_i(t)}) = (\varphi \circ \Phi_{-t})|_{\partial\Omega_i(t)},
\end{equation*}
with the final equality because $\Phi_t$ maps $\partial\Omega_i(0)$ to $\partial\Omega_i(t)$.

For an arbitrary  $u \in H^1\bigl(\Omega_i(0)\bigr)$ take a sequence $\{u_n\} \subset C^\infty(\overline{\partial\Omega_i(0)})$ converging to $u$. One then has the equality
\begin{equation*}
T_{\partial\Omega_i(t)}(\phi_t u_n)  = \psi_t(T_{\partial\Omega_i(0)} u_n).
\end{equation*}
The trace operators $T_{\partial\Omega_i(s)}: H^1\bigl(\Omega_i(s)\bigr) \to H^{1/2}\bigl(\partial\Omega_i(s)\bigr)$, $s=0,t$, are continuous and, by~\cref{lem:H1Compatibility,lem:compatH1/2}, the same holds for $\phi_t:H^1\bigl(\Omega_i(0)\bigr) \to H^1\bigl(\Omega_i(t)\bigr)$ and $\psi_t:H^{1/2}\bigl(\partial\Omega_i(0)\bigr) \to H^{1/2}\bigl(\partial\Omega_i(t)\bigr)$. Hence, we obtain that the lemma's first equality holds in $H^{1/2}\bigl(\partial\Omega_i(t)\bigr)$. The second equality follows by~\cref{lem:compatLambda} and the same line of argumentation. The only difference is that the restrictions are made to $\Gamma$ instead of $\partial\Omega_i$, and $\{u_n\} \subset \{\varphi\in C^\infty(\overline{\partial\Omega_i(0)}):\restr{\varphi}{\partial\Omega_i(0)\backslash\Gamma(0)}=0\}$.
\end{proof}

\begin{lemma}
If \cref{ass:phi} holds then $\bigl(H^1_0(\Omega), \phi\bigr)$, $\bigl(H^1_0(\Omega_i), \phi\bigr)$, and $(V_i, \phi)$ are all compatible pairs. 
\end{lemma}

\begin{proof}
We prove the $V_i$ case as the others follow directly by combining~\cref{lem:H1Compatibility,lem:commutationOfphiAndTrace}. For $u\in V_i(0)\subset H^1\bigl(\Omega_i(0)\bigr)$ we have $\phi_t u\in H^1\bigl(\Omega_i(t)\bigr)$ and $\restr{(T_{\partial\Omega_i(0)} u_n)}{\partial\Omega_i(0)\backslash\Gamma(0)}=0$. As $\{\varphi\in C^\infty(\overline{\partial\Omega_i(0)}):\restr{\varphi}{\partial\Omega_i(0)\backslash\Gamma(0)}=0\}$ is dense in~$V_i(0)$, we can choose a sequence $\{u_n\}$ in this dense subset such that it converges to~$u$. \cref{lem:commutationOfphiAndTrace} then implies that 
\begin{equation}\label{eq:compatVi}
\begin{aligned}
\restr{(T_{\partial\Omega_i(t)}\phi_t u_n\bigr)}{\partial\Omega_i(t)\backslash\Gamma(t)} & = \restr{(\psi_tT_{\partial\Omega_i(0)} u_n)}{\partial\Omega_i(t)\backslash\Gamma(t)}\\				          &= (\restr{u_n}{\partial\Omega_i(0)})\circ(\restr{\Phi_{-t}}{\partial\Omega_i(t)\backslash\Gamma(t)})=\restr{u_n}{{\partial\Omega_i(0)\backslash\Gamma(0)}}=0.	
\end{aligned}		                                                            
\end{equation}
Here, the second-to-last equality follows as $\Phi_t$ maps $\partial\Omega_i(0)\backslash\Gamma(0)$ to $\partial\Omega_i(t)\backslash\Gamma(t)$. The map $T_{\partial\Omega_i(t)}\phi_t:H^1\bigl(\Omega_i(0)\bigr) \to H^{1/2}\bigl(\partial\Omega_i(t)\bigr)$ is continuous by~\cref{lem:H1Compatibility}, and
taking the limit in~\cref{eq:compatVi} yields that $\restr{(T_{\partial\Omega_i(t)}\phi_t u\bigr)}{\partial\Omega_i(t)\backslash\Gamma(t)} =0$ in $L^2\bigl(\partial\Omega_i(t)\backslash\Gamma(t)\bigr)$. In conclusion, $\phi$ maps $V_i(0)$ into $V_i(t)$, and the same reasoning can be made for~$\phi_-$. The bounds and measurability stated in~\cref{def:def_compatibility} follow as for $H^1\bigl(\Omega_i(t)\bigr)$ in~\cref{lem:H1Compatibility}. Thus, $(V_i, \phi)$ is a compatible pair.
\end{proof}

Due to the results derived in~\cref{sec:abs,sec:comp}, the maps $(\phi,\psi)$ are isomorphisms with equivalent norms on all the Sobolev--Bochner spaces and their intersections appearing in the rest of the paper, with the one exception of the space $H_0^1\bigr(\mathbb{R}_+;H^{-1}(\Omega_i(0))\bigr)$.
 
\section{Weak formulation and existence of solutions}\label{sec:weak}

We now address existence of weak solutions for \cref{eq:strong}. First, we set the stage. Let $M$ once more play the role of $\Omega_i,\Omega$, where $M(0)\subset B_r$ and $M(t)\subset B_R$. Furthermore,~$g$ will either denote~$f$, the right-hand side of~\cref{eq:strong}, or its restriction $f_i$ to~$\Omega_i$. Throughout this section $X$ will denote 
$H^1_0(\Omega)$ or $H^1_0(\Omega_i)$, and 
\begin{equation*}
U=H^1_{L^2(M)}(\mathbb{R}_+)\cap L^2_{X}(\mathbb{R}_+).
\end{equation*}
\begin{definition}
    The weak form of~\cref{eq:strong} and its counterpart on $\Omega_i$ (which all have homogeneous initial conditions) can be formulated as
finding $u\in L^2_{X}(\mathbb{R}_+)$ such that 
\begin{equation}\label{eq:weak}
a(u,v)=d(\phi_-u,\phi_-v)+c(u,v)=\langle g,v\rangle\quad\text{for all }v\in U,
\end{equation}
where
\begin{equation*}
d(u,v)=-\int_{\mathbb{R}_+}\int_{M(0)}u\,\partial_tv\,|J_t|\,\mathrm{d} x_0\mathrm{d}t,
\quad\text{and}\quad
c(u,v)=\int_{\mathbb{R}_+}\int_{M(t)}\alpha\nabla u\cdot\nabla v+\beta uv\,\mathrm{d}x_t\mathrm{d}t.
\end{equation*}
\end{definition}
The weak form can be derived as follows. First, note that we formally have the identity
\begin{equation}\label{eq:dtM0}
\frac{\mathrm{d}}{\mathrm{d}t} u\bigl(t,\Phi_t(x)\bigr)=(\partial_t u+\nabla u\cdot\mathbf{w})\bigl(t,\Phi_t(x)\bigr)
\quad\Leftrightarrow\quad\frac{\mathrm{d}}{\mathrm{d}t}\phi_{-t}u = \phi_{-t}\dot{u}
\end{equation}
and Jacobi's formula
\begin{equation}\label{eq:diffOfJ}
\frac{d}{dt}J_t(x)= \nabla \cdot \mathbf{w}\bigl(t,\Phi_t(x)\bigr) J_t(x) 
\quad\Leftrightarrow\quad \frac{\mathrm{d}}{\mathrm{d}t}J_t = \phi_{-t}(\nabla \cdot \mathbf{w})J_t.
\end{equation}
Observe that if Jacobi's formula holds then it implies that $\mathrm{d}|J_t|/\mathrm{d}t = \phi_{-t}(\nabla \cdot \mathbf{w})|J_t|$. 
Consider sufficiently regular functions $u,v$ such that $\phi_{-t}\bigl(u(t)v(t)\bigr)$ decays sufficiently rapidly as $t$ tends to infinity and $\restr{v(t)}{S(t)}=0$. 
    Integration by parts in time together with~\cref{eq:dtM0,eq:diffOfJ} then gives 
\begin{equation*}
\begin{aligned}
&\int_{\mathbb{R}_+}\int_{M(t)}\dot{u}v\,\mathrm{d}x_t\mathrm{d}t
=\int_{\mathbb{R}_+}\int_{M(0)}\phi_{-t}\dot{u}\,\phi_{-t}v|J_t| \,\mathrm{d}x_0\mathrm{d}t\\
&\quad=\int_{\mathbb{R}_+}\int_{M(0)}\frac{\mathrm{d}}{\mathrm{d}t}(\phi_{-t}u)\,\phi_{-t}v|J_t| \,\mathrm{d}x_0\mathrm{d}t\\
&\quad=-\int_{\mathbb{R}_+}\int_{M(0)}\phi_{-t}u\,\frac{\mathrm{d}}{\mathrm{d}t}(\phi_{-t}v|J_t|)\,\mathrm{d} x_0\mathrm{d}t\\
&\qquad+\lim_{\tau\to\infty}\int_{M(0)}\restr{\phi_-(uv)}{t=\tau}|J_\tau|\,\mathrm{d} x_0-\int_{M(0)}\restr{(uv)}{t=0}|J_0|\,\mathrm{d} x_0\\
&\quad=-\int_{\mathbb{R}_+}\int_{M(0)}\phi_{-t}u\,\bigl(\frac{\mathrm{d}}{\mathrm{d}t}(\phi_{-t}v)|J_t|+\phi_{-t}(\nabla \cdot \mathbf{w})|J_t|\bigr)\,\mathrm{d} x_0\mathrm{d}t\\
&\qquad-\int_{M(0)}\restr{(uv)}{t=0}\,\mathrm{d} x_0\\
&\quad=-\int_{\mathbb{R}_+}\int_{M(0)}\phi_{-t}u\,\frac{\mathrm{d}}{\mathrm{d}t}(\phi_{-t}v)|J_t|\,\mathrm{d} x_0\mathrm{d}t
-\int_{\mathbb{R}_+}\int_{M(t)}(\nabla \cdot \mathbf{w})uv\,\mathrm{d} x_t\mathrm{d}t\\
&\qquad-\int_{M(0)}\restr{(uv)}{t=0}\,\mathrm{d} x_0.
\end{aligned}
\end{equation*}
The above equality and integration by parts in space yields 
\begin{align}\label{eq:initial}
\nonumber&\int_{\mathbb{R}_+}\int_{M(t)}\bigl(\dot{u}-\nabla\cdot(\alpha\nabla u)+(\nabla \cdot \mathbf{w}+\beta)u\bigr)v\,\mathrm{d}x_t\mathrm{d}t\\
\nonumber&\quad=-\int_{\mathbb{R}_+}\int_{M(0)}\phi_{-t}u\,\frac{\mathrm{d}}{\mathrm{d}t}(\phi_{-t}v)|J_t|\,\mathrm{d} x_0\mathrm{d}t
-\int_{\mathbb{R}_+}\int_{M(t)}(\nabla \cdot \mathbf{w})uv\,\mathrm{d} x_t\mathrm{d}t\\
\nonumber&\qquad-\int_{M(0)}\restr{(uv)}{t=0}\,\mathrm{d} x_0\\
\nonumber&\qquad+\int_{\mathbb{R}_+}\int_{M(t)}\alpha\nabla u\cdot\nabla v+(\nabla \cdot \mathbf{w}+\beta)uv\,\mathrm{d}x_t\mathrm{d}t-\int_{\mathbb{R}_+}\int_{\partial M(t)} \alpha \nabla u\cdot n \,v\,\mathrm{d}s_t\mathrm{d}t\\
\nonumber&\quad=-\int_{\mathbb{R}_+}\int_{M(0)}\phi_{-t}u\,\frac{\mathrm{d}}{\mathrm{d}t}(\phi_{-t}v)|J_t|\,\mathrm{d} x_0\mathrm{d}t
+\int_{\mathbb{R}_+}\int_{M(t)}\alpha\nabla u\cdot\nabla v+\beta uv\,\mathrm{d}x_t\mathrm{d}t\\
&\quad-\int_{M(0)}\restr{(uv)}{t=0}\,\mathrm{d} x_0,
\end{align}
which justifies the definition.

Now, in order to have a well defined weak problem, we assume the following.
\begin{assumption}\label{ass:eqdata}
The problem data $(\alpha,\beta,\mathbf{w},f)$ in~\cref{eq:strong} fulfills the properties
\begin{enumerate}[label=({\roman*})]
\item $\alpha,\beta \in L^{\infty}\bigl(\mathbb{R}_+;L^\infty(B_R)\bigr)$, and $\mathbf{w}\in L^\infty\bigl(\mathbb{R}_+; W^{1,\infty}(B_R,\mathbb{R}^n)\bigr)$;
\item $\alpha(t,x)\geq c>0$ a.e.\ $(t,x)\in \mathbb{R}_+\times B_R$;
\item there exists a constant $c>0$ such that
\begin{equation*}
1/2\,\nabla \cdot \mathbf{w}(t,x)+\beta(t,x) \geq c,
\end{equation*}
for a.e.\ $(t,x)\in \mathbb{R}_+\times B_R$;
\item $f\in L^2_{H^{-1}(\Omega)}(\mathbb{R}_+)$ and there exist $f_i\in L^2_{V_i^*}(\mathbb{R}_+)$, $i=1, 2$ such that
\begin{equation*}
    \langle f, v\rangle=\bigl\langle f_1, \phi\bigl(\restr{(\phi_-v)}{\R_+\times\Omega_1(0)}\bigr)\bigr\rangle +\bigl\langle f_2, \phi\bigl(\restr{(\phi_-v)}{\R_+\times\Omega_2(0)}\bigr)\bigr\rangle
\end{equation*}
for all $v\in L^2_{H^1_0(\Omega)}(\R_+)$.
\end{enumerate}
\end{assumption}
Note that if~\cref{ass:phi,ass:eqdata} hold then the right-hand side of~\cref{eq:diffOfJ} is an element in $L^\infty\bigl(\mathbb{R}_+;L^\infty(B_r)\bigr)$, i.e., $J,|J|,1/|J|\in W^{1,\infty}\bigl(\mathbb{R}_+; L^\infty(B_r)\bigr)$. To avoid a few technicalities we also assume the following. 
\begin{assumption}\label{ass:tech}
The map $|J|$ is an element in $L^\infty\bigl(\mathbb{R}_+; W^{1,\infty}(B_r)\bigr)$.
\end{assumption}
This assumption also yields that $1/|J|\in L^\infty\bigl(\mathbb{R}_+; W^{1,\infty}(B_r)\bigr)$. Next, we introduce
\begin{equation}\label{eq:test}
\begin{aligned}
&\mathcal{D}=\bigl\{u\in L^2_{L^2(M)}(\mathbb{R}_+):\phi_-u=\restr{v}{\mathbb{R}_+\times M(0)}\text{ with }v\in C^{\infty}_0\bigl(\mathbb{R}\times B_r\bigr)\big\}\quad\text{and}\\ 
&\mathcal{D}_0=\bigl\{u\in L^2_{L^2(M)}(\mathbb{R}_+):\phi_-u=\restr{v}{\mathbb{R}_+\times M(0)}\text{ with }v\in C^{\infty}_0\bigl(\mathbb{R}\times M(0)\bigr)\big\}.
\end{aligned}
\end{equation}
\begin{lemma}\label{lem:density1}
If~\cref{ass:phi} holds, then $\mathcal{D}$ and $\mathcal{D}_0$ are dense in $H^1_{L^2(M)}(\mathbb{R}_+)$ and
$L^2_X(\mathbb{R}_+)$, respectively. 
\end{lemma}
\begin{proof}
By~\cref{ass:phi}, we have that $\phi\colon H^1\bigl(\mathbb{R}_+;L^2(M(0))\bigr)\to H^1_{L^2(M)}(\mathbb{R}_+)$ is an isomorphism with an equivalence of norms, and it is therefore sufficient to prove that $\phi_-(\mathcal{D})$ is dense in $H^1\bigl(\mathbb{R}_+;L^2(M(0))\bigr)$. To this end, observe that $C^\infty_1=\bigl\{u:\phi_-u=\restr{v}{\mathbb{R}_+}\text{ with }v\in C^{\infty}_0\bigl(\mathbb{R}\bigr)\big\}$
is dense in $H^1(\mathbb{R}_+)$ and $C^\infty_2=C^\infty_0\bigl(M(0)\bigr)$ is dense in $L^2\bigl(M(0)\bigr)$.  We then have that the
algebraic tensor space $C^\infty_1\otimes C^\infty_2\subset \phi_-(\mathcal{D})$ is dense in the completion of $H^1(\mathbb{R}_+)\otimes L^2\bigl(M(0)\bigr)$. The latter space is isomorphic with an equivalence of norms to $H^1\bigl(\mathbb{R}_+;L^2(M(0))\bigr)$, see~\cite[Section~4]{EEEH24} for details. Hence, $\mathcal{D}$ is dense in $H^1_{L^2(M)}(\mathbb{R}_+)$. The density of $D_0$ in $L^2_X(\mathbb{R}_+)$ follows by the same argument.
\end{proof}
\begin{lemma}\label{lem:dmon}
If~\cref{ass:phi,ass:eqdata} hold then 
\begin{equation*}
d(\phi_{-}u,\phi_{-}u)\geq \int_{\mathbb{R}_+}\int_{M(t)}(1/2\nabla \cdot \mathbf{w})u^2\, \mathrm{d}x_t\mathrm{d}t\quad\text{for all }u\in H^1_{L^2(M)}(\mathbb{R}_+).
\end{equation*}
\end{lemma}
\begin{proof}
For an arbitrary $v\in \phi_-(\mathcal{D})$ we have 
\begin{equation*}
\begin{aligned}
d(v,v) &=-\int_{\mathbb{R}_+}\int_{M(0)}v\,\partial_tv\,|J_t|\,\mathrm{d} x_0\mathrm{d}t\\
&=\int_{\mathbb{R}_+}\int_{M(0)}\partial_t(v|J_t|)\,v\,\mathrm{d} x_0\,\mathrm{d}t - \lim_{T\to\infty}\int_{M(0)}\phi_{-T}v^2|J_T|\,\mathrm{d} x_0+\int_{M(0)}v^2\,\mathrm{d} x_0\\
&=\int_{\mathbb{R}_+}\int_{M(0)}v\,\partial_tv\,|J_t|\,\mathrm{d} x_0\,\mathrm{d}t+\int_{\mathbb{R}_+}\int_{M(0)}v^2\phi_{-t}(\nabla \cdot \mathbf{w})|J_t|\mathrm{d} x_0\mathrm{d}t+\int_{M(0)}v^2\,\mathrm{d} x_0\\
&\geq -d(v,v)+\int_{\mathbb{R}_+}\int_{M(0)}\phi_{-t}(\nabla \cdot \mathbf{w})v^2|J_t|\,\mathrm{d} x_0\mathrm{d}t.
\end{aligned}
\end{equation*}
Here, the limit term is zero as $v$ has a compact support in time. With $u=\phi v\in\mathcal{D}$ the above inequality is equivalent to 
\begin{equation*}
d(\phi_-u,\phi_-u)\geq \int_{\mathbb{R}_+}\int_{M(t)}(1/2\nabla \cdot \mathbf{w})u^2\, \mathrm{d}x_t\mathrm{d}t.
\end{equation*}
The bound is also valid for all $u\in H^1_{L^2(M)}(\mathbb{R}_+)$, as the bilinear form $d(\phi_{-t}\bigr(\cdot),\phi_{-t}(\cdot)\bigl)\colon H^1_{L^2(M)}(\mathbb{R}_+)\times H^1_{L^2(M)}(\mathbb{R}_+)\to\mathbb{R}$ is continuous and $\mathcal{D}$ is dense in $H^1_{L^2(M)}(\mathbb{R}_+)$  by~\cref{lem:density1}.
\end{proof}

Before we proceed, we introduce the trace operators on the space-time cylinder $\mathbb{R}_+\times M(0)$. By considering tensor operators, compare~\cite[Section~4]{EEEH24}, we can extend the spatial trace operators on $M(0)$ (see~\cref{sec:comp}) to
\begin{equation}\label{eq:tracespaacetime}
\begin{aligned}
            T_{\partial M(0)}&\colon L^2\bigl(\mathbb{R_+};H^1(M(0))\bigr)\rightarrow L^2\bigl(\mathbb{R_+};H^{1/2}(\partial M(0))\bigr),\\
            T_{i,0}&\colon L^2\bigl(\mathbb{R_+};V_i(0)\bigr)\rightarrow L^2\bigl(\mathbb{R_+};\Lambda(0)\bigr),
\end{aligned}
\end{equation}
which once more are linear, bounded, and surjective. Furthermore, we have the equality
\begin{equation*}
T_{i,0}v =\restr{(T_{\partial\Omega_i(0)}v)}{\mathbb{R}_+\times \Gamma(0)}\quad\text{for all }v\in L^2\bigl(\mathbb{R_+};V_i(0)\bigr)
\end{equation*}
and, by~\cite[Lemma~4.2]{EEEH24}, the identifications 
\begin{equation}\label{eq:L2Videntities}
\begin{aligned}
L^2\bigl(\mathbb{R_+};X(0)\bigr)&=\bigl\{v\in L^2\bigl(\mathbb{R_+};H^1(M(0))\bigr):T_{\partial M(0)} v=0\bigr\}\quad\text{and}\\
L^2\bigl(\mathbb{R_+};V_i(0)\bigr)&=\bigl\{v\in L^2\bigl(\mathbb{R_+};H^1(\Omega_i(0))\bigr):\restr{(T_{\partial\Omega_i(0)}v)}{\mathbb{R}_+\times (\partial \Omega_i(0)\backslash \Gamma(0))}=0\bigr\}.
\end{aligned}
\end{equation}

\begin{lemma}\label{lem:Jvtest}
If \cref{ass:phi,ass:eqdata,ass:tech} hold and $v\in \phi_-(U)$ then $|J|v$ and $1/|J|v$ are also elements in $\phi_-(U)$.
\end{lemma}
\begin{proof}
If $v\in \phi_-(U)$ then, by~\cref{ass:tech} and the chain rule, one trivially obtains that $|J|v$ is an element in $H^1\bigl(\mathbb{R_+};L^2(M(0))\bigr)\cap L^2\bigl(\mathbb{R_+};H^1(M(0))\bigr)$ and that the map $v\mapsto |J|v$ is continuous in $L^2\bigl(\mathbb{R_+};H^1(M(0))\bigr)$. It remains to verify that $|J|v\in L^2\bigl(\mathbb{R_+};X(0)\bigr)$. As $\mathcal{D}_0\subset U$ is dense in $L^2_X(\mathbb{R}_+)$ and $|J|\in C(\mathbb{R}_+\times \mathbb{R}^n, \mathbb{R})$, we can choose a sequence $\{v_n\}\subset \phi_-(\mathcal{D}_0)$ that converges to $v$ in $L^2\bigl(\mathbb{R_+};X(0)\bigr)$ and obtain that
\begin{equation*}
T_{\partial M(0)}(|J|v)=\lim_{n\to\infty} T_{\partial M(0)}(|J|v_n)
=\lim_{n\to\infty} \restr{( |J| v_n)}{\mathbb{R}_+\times\partial M(0)}=0.
\end{equation*}
Hence, $|J|v$ is also an element in $L^2\bigl(\mathbb{R_+};X(0)\bigr)$, i.e., $|J|v\in \phi_-(U)$. The very same argumentation also holds for $1/|J|v$.
\end{proof}

We can now prove the existence of a solution to~\cref{eq:weak} with homogeneous Dirichlet boundary conditions. The proof closely follows~\cite[Lemma~2.3]{costabel90} and is based on Lions' projection lemma, see~\cite{Lions1} or~\cite[Lemma~2.4]{Scherzer} for an English proof. 
\begin{theorem}\label{lem:homexist}
If~\cref{ass:phi,ass:eqdata,ass:tech} hold, then for every $g\in L^2_{X^*}(\mathbb{R}_+)$ there exists a solution $u\in L^2_X(\mathbb{R}_+)$ to~\cref{eq:weak} such that $\phi_-u\in H_0^1(\mathbb{R}_+;X(0)^*)$
and
\begin{equation}\label{eq:hombound}
\bigl(\|u\|^2_{L^2_{X}(\mathbb{R}_+)}+\|\partial_t(\phi_-u)\|^2_{L^2(\mathbb{R}_+;X(0)^*)}\bigr)^{1/2}\leq C \|g\|_{L^2_{X^*}(\mathbb{R}_+)}.
\end{equation}
\end{theorem}
\begin{proof}
Consider the bilinear form $a\colon L^2_X(\mathbb{R}_+)\times U\to\mathbb{R}$ and observe that $u\mapsto a(u,v)$ is continuous on $L^2_X(\mathbb{R}_+)$ for every fixed $v\in U$. For $u\in U$ we have, by~\cref{ass:eqdata,lem:dmon}, that
\begin{equation*}
a(u,u)\geq \int_{\mathbb{R}_+}\int_{M(t)}\alpha |\nabla u|^2+(1/2\nabla \cdot \mathbf{w}+\beta)u^2\, \mathrm{d}x_t\mathrm{d}t\geq c\|u\|^2_{L^2_{H^1(M)}(\mathbb{R}_+)}.
\end{equation*}
Furthermore, $\mathcal{D}_0 \subset U$ is dense in $L^2_X(\mathbb{R}_+)$ via \cref{lem:density1}.
These properties of $a$ yield that the hypothesis of Lions' projection lemma is fulfilled, i.e., there exists a solution $u\in L^2_X(\mathbb{R}_+)$ to~\cref{eq:weak} such that 
\begin{equation}\label{eq:uL2Xbound}
\|u\|_{L^2_{X}(\mathbb{R}_+)}\leq C \|g\|_{L^2_{X^*}(\mathbb{R}_+)}.
\end{equation}
Here, we have used that $L^2_{X}(\mathbb{R}_+)^*\cong L^2_{X^*}(\mathbb{R}_+)$, see~\cref{sec:abs}. 

It remains to show that we have the higher regularity $\phi_-u\in H_0^1(\mathbb{R}_+;X(0)^*)$. To this end, let $\varphi\in C^\infty_0(\mathbb{R}_+)$ and $w\in X(0)$. With $v=w\varphi$ we have that $\phi v$ and $|J_-|\phi v=\phi(1/|J|v)$ are elements in $U$, the latter by~\cref{lem:Jvtest}. Furthermore, 
\begin{align}\label{eq:weaktimederivative1}
\nonumber&-\int_{\mathbb{R}_+}\Bigl(\int_{M(0)}\phi_{-t}u(t)\, w\,\mathrm{d} x_0\Bigr)\varphi'(t)\,\mathrm{d}t\\
\nonumber&\quad=-\int_{\mathbb{R}_+}\int_{M(0)}\phi_{-t}u\,\bigl(\partial_t (1/|J_t|v)-\partial_t(1/|J_t|)v\bigr)|J_t|\,\mathrm{d} x_0\mathrm{d}t\\
\nonumber&\quad=d\bigl(\phi_{-}u, \phi_-(|J_-|\phi v)\bigr)+\int_{\mathbb{R}_+}\int_{M(0)}\partial_t(1/|J_t|)\,|J_t|\,\phi_{-t}u\,v\,\mathrm{d} x_0\mathrm{d}t\\
\nonumber&\quad=\langle g,|J_-|\phi v\rangle_{L^2_{X^*}(\mathbb{R}_+)\times L^2_{X}(\mathbb{R}_+)}-c(u,|J_-|\phi v)+\int_{\mathbb{R}_+}\int_{M(0)}\partial_t(1/|J_t|)\,|J_t|\,\phi_{-t}u\,v\,\mathrm{d} x_0\mathrm{d}t\\
\nonumber&\quad=\int_{\mathbb{R}_+}\Bigl(\langle g(t),\phi_t(1/|J_t|w)\rangle_{X(t)^*\times X(t)}\\
&\qquad-\int_{M(t)}\alpha\nabla u\cdot\nabla \phi_t(1/|J_t|w) +\beta u\phi_t(1/|J_t|w)\, \mathrm{d}x_t\\
\nonumber&\qquad+\int_{M(0)}\partial_t(1/|J_t|)\,|J_t|\,\phi_{-t}u\,w\,\mathrm{d} x_0\Bigr)\varphi(t)\mathrm{d}t\\
\nonumber&\quad=\int_{\mathbb{R}_+}\langle p(t), w\rangle_{X(0)^*\times X(0)}\,\varphi(t)\,\mathrm{d}t,
\end{align}
where $p\in  L^2\bigl(\mathbb{R}_+;X(0)^*\bigr)$ and 
\begin{align}\label{eq:weaktimederivative2}
\nonumber\|p\|_{L^2(\mathbb{R}_+;X(0)^*)}& \leq C\|1/J\|_{L^\infty(\mathbb{R}_+; W^{1,\infty}(B_r))}\|g\|_{ L^2_{X^*}(\mathbb{R}_+)}\\
\nonumber&\quad+C\|1/J\|_{L^\infty(\mathbb{R}_+; W^{1,\infty}(B_r))}\|u\|_{L^2_X(\mathbb{R}_+)}\\
&\quad+C\|1/J\|_{W^{1,\infty}(\mathbb{R}_+; L^\infty(B_r))}\|J\|_{L^\infty(\mathbb{R}_+; L^\infty(B_r))}\|u\|_{L^2_{L^2(M)}(\mathbb{R}_+)}\\
\nonumber&\leq C\|g\|_{ L^2_{X^*}(\mathbb{R}_+)}.
\end{align}
The last bound follows by~\cref{eq:uL2Xbound}. Hence, $\partial_t\phi_- u\in L^2\bigl(\mathbb{R}_+;X(0)^*\bigr)$, i.e., $\phi_-u\in H^1(\mathbb{R}_+;X(0)^*)$, and~\cref{eq:hombound} follows by the bounds above. As 
\begin{equation*}
L^2\bigl(\mathbb{R}_+;X(0))\bigr)\cap H^1(\mathbb{R}_+;X(0)^*)\hookrightarrow C\bigl(\mathbb{R}^0_+;L^2(M(0))\bigr)
\end{equation*}
we have that $\restr{(\phi_-u)}{t=0}\in L^2\bigl(M(0)\bigr)$. Combining~\cref{eq:initial,eq:weak} yields $\restr{(\phi_-u)}{t=0}=0$, and thus $\phi_-u\in H_0^1(\mathbb{R}_+;X(0)^*)$, compare with~\cite[Equation~2.2]{costabel90} and the proof of~\cite[Lemma~2.3]{costabel90}.
\end{proof}

\section{Temporal \texorpdfstring{$H^{1/2}$-setting}{H\^{½}-setting} for evolving domains}\label{sec:Honehalf}

In the rest of the paper we will make use of the Sobolev--Bochner spaces stated in~\cref{tab:spaces}. Here, 
\begin{equation*}
	\begin{gathered}
	    H^{s}(I)=\{u\in L^2(I):  \|u\|_{H^{s}(I)}<\infty\}\quad\text{with}\quad
	    \|u\|^2_{H^{s}(I)} =
	                 |u|_{H^{s}(I)}^2
	                        +\|u\|^2_{L^2(I)}\\
	    \text{and}\quad|u|^2_{H^{s}(I)} =
	                 \int_{I}\int_{I}\frac{\abs{u(\tau)-u(t)}^2}{|\tau-t|^{1+2s}}\,\mathrm{d}\tau\,\mathrm{d} t,
	\end{gathered}
\end{equation*}
for $s=1/2$ or $s=1/4$ and on the time intervals $I=\mathbb{R}_+$ or $I=\mathbb{R}$. Furthermore, $H_{(0,.)}^{1/2}(\mathbb{R}_+)$ is the temporal Lions--Magenes space, i.e., 
\begin{equation*}
	   H_{(0,.)}^{1/2}(\mathbb{R}_+)=\{u\in L^2(\mathbb{R}_+): e_{\mathbb{R}}u\in H^{1/2}(\mathbb{R})\}\quad\text{with}\quad
	    \|u\|_{H_{(0,.)}^{1/2}(\mathbb{R}_+)}=\|e_{\mathbb{R}}u\|_{H^{1/2}(\mathbb{R})},
\end{equation*}
where  $e_{\mathbb{R}}$ denotes the extension by zero from $\mathbb{R}_+$ to $\mathbb{R}$. We will also use the notation 
\begin{equation*}
\begin{gathered}
a_i(u,v) =d_i(\phi_-u,\phi_-v)+c_i(u,v), \quad\text{where}\\
d_i(u,v)= -\int_{\mathbb{R}_+}\int_{\Omega_i(0)}u\,\partial_tv\,|J_t|\,\mathrm{d} x_0\mathrm{d}t\quad\text{and} 
\quad c_i(u,v)=\int_{\mathbb{R}_+}\int_{\Omega_i(t)}\alpha\nabla u\cdot\nabla v+\beta uv\,\mathrm{d}x_t\mathrm{d}t.
\end{gathered}
\end{equation*}
We denote the corresponding bilinear forms on the whole domain $\Omega$ by~$a,d,c$.
\begin{table}[t]
\def\arraystretch{1.7}
\normalsize
\centering
\begin{tabular}{lcl}
\hline 
 \multicolumn{3}{l}{$Q_i=H_0^1\bigr(\mathbb{R}_+;H^{-1}(\Omega_i(0))\bigr)\cap L^2\bigl(\mathbb{R}_+;V_i(0)\bigr)$}\\
	$U_i=H^1_{L^2(\Omega_i)}(\mathbb{R}_+)\cap L^2_{V_i}(\mathbb{R}_+)$
	& \qquad& 
	$U_i^0=H^1_{L^2(\Omega_i)}(\mathbb{R}_+)\cap L^2_{H^1_0(\Omega_i)}(\mathbb{R}_+)$\\
	$W=H^{1/2}_{(0,\cdot)\,L^2(\Omega)}(\mathbb{R}_+)\cap L^2_{H^1_0(\Omega)}(\mathbb{R}_+)$
	& \qquad&
        $\tilde{W}=H^{1/2}_{L^2(\Omega)}(\mathbb{R}_+)\cap L^2_{H^1_0(\Omega)}(\mathbb{R}_+)$\\
        $W_i=H^{1/2}_{(0,\cdot)\,L^2(\Omega_i)}(\mathbb{R}_+)\cap L^2_{V_i}(\mathbb{R}_+)$
        & \qquad&
        $\tilde{W}_i=H^{1/2}_{L^2(\Omega_i)}(\mathbb{R}_+)\cap L^2_{V_i}(\mathbb{R}_+)$\\
        $W_i^0=H^{1/2}_{(0,\cdot)\,L^2(\Omega_i)}(\mathbb{R}_+)\cap L^2_{H^1_0(\Omega_i)}(\mathbb{R}_+)$  
        & \qquad&
        $\tilde{W}_i^0=H^{1/2}_{L^2(\Omega_i)}(\mathbb{R}_+)\cap L^2_{H^1_0(\Omega_i)}(\mathbb{R}_+)$\\
 \multicolumn{3}{l}{ $Z=H^{1/4}_{L^2(\Gamma)}(\mathbb{R}_+)\cap L^2_{\Lambda}(\mathbb{R}_+)$}\\[5pt]       
\hline
\end{tabular}
\caption{Sobolev--Bochner spaces used in~\cref{sec:Honehalf,sec:trans,sec:SP}.}\label{tab:spaces}
\end{table}

As already stated in the introduction, analyzing the equivalence between the weak form of the original parabolic equation~\cref{eq:weak} and the transmission problem~\cref{eq:trans} is difficult in the space of solutions with temporal regularity of the form $\phi_-u\in H_0^1\bigr(\mathbb{R}_+;H^{-1}(\Omega_i(0))\bigr)$, compare with the weak solution in~\cref{lem:homexist}. Instead we observe that the abstract interpolation result~\cite[Equation 2.24]{costabel90} together with the identification $[H^1(\Omega_i(0)),H^{-1}(\Omega_i(0))]_{1/2}\cong L^2(\Omega_i(0))$, compare with~\cite[Lemma~12.1]{lionsmagenes1}, gives 
\begin{equation*}
\bigl[H_0^1\bigr(\mathbb{R}_+;H^{-1}(\Omega_i(0))\bigr),L^2\bigl(\mathbb{R}_+;H^1(\Omega_i(0))\bigr)\bigr]_{\frac12}\cong H^{1/2}_{(0,\cdot)}\bigl(\mathbb{R}_+;L^2(\Omega_i(0))\bigr).
\end{equation*}
Hence, $Q_i$ is embedded into $H^{1/2}_{(0,\cdot)}\bigl(\mathbb{R}_+;L^2(\Omega_i(0))\bigr)$, and we obtain 
\begin{equation}\label{eq:keyinterpolate}
Q_i\hookrightarrow \phi_-(W_i).
\end{equation}
The embedding $H_0^1\bigr(\mathbb{R}_+;H^{-1}(\Omega_i(0))\bigr)\cap L^2\bigl(\mathbb{R}_+;H^1_0(\Omega_i(0))\bigr)\hookrightarrow \phi_-(W^0_i)$ also holds true, see~\cite[Equation 2.25]{costabel90}. One possibility is therefore to consider the solution space, or trial space,~$W_i$ together with the test space~$\tilde{W}_i$. 

In preparation for the analysis of the transmission problem, we prove the existence of a unique solution~$u\in W_i$ to to the parabolic equation on $\Omega_i$ with \emph{inhomogeneous} Dirichlet boundary conditions. To this end, observe that the trace operators~\cref{eq:tracespaacetime} can be restricted as 
\begin{equation*}
\begin{aligned}
T_{\partial \Omega_i(0)}&\colon H_{(0,\cdot)}^{1/2}\bigl(\mathbb{R_+}; L^2(\Omega_i(0))\bigr)\cap L^2\bigl(\mathbb{R_+};H^1(\Omega_i(0))\bigr)\rightarrow\\ 
&\qquad H^{1/4}\bigl(\mathbb{R_+}; L^2(\partial \Omega_i(0))\bigr)\cap L^2\bigl(\mathbb{R_+};H^{1/2}(\partial \Omega_i(0))\bigr)\quad\text{and}\\
            T_{i,0}&\colon \phi_-(W_i)\rightarrow \psi_-(Z),
\end{aligned}
\end{equation*}
where the new operators are all linear, bounded, and surjective, see~\cite[Lemma 2.4]{costabel90} and~\cite[Lemma~4.4]{EEEH24}. We also recapitulate the existence result for the heat equation on~$\Omega_i(0)$ with inhomogeneous boundary conditions, see \cite[Corollary~2.11]{costabel90}.
\begin{lemma}\label{lem:nonhomstatexist}
If~\cref{ass:phi} hold then for every $\eta \in \psi_-(Z)$ there exists a unique solution 
$u\in Q_i$ to the heat equation
\begin{equation}\label{eq:heateq}
\ell(u,v)=\int_{\mathbb{R}_+}\int_{\Omega_i(0)} -u\partial_t v+\nabla u\cdot\nabla v\,\mathrm{d}x_0\mathrm{d}t=0\quad\text{for all }v\in\phi_-(U_i^0),
\end{equation}
such that $T_{i,0}u=\eta$.
\end{lemma}
Note that~\cref{lem:nonhomstatexist} implies that the restricted operator $\restr{T_{i,0}}{L}$ is bijective, where 
\begin{equation*}
L=\{u\in Q_i:l(u,v)=0\text{ for all } v\in\phi_-(U_i^0)\}
\end{equation*}
is a closed subset of $\phi_-(W_i)$.  The open mapping theorem then yields that the solution operator 
\begin{equation*}
R_{i,0}=(\restr{T_{i,0}}{L})^{-1}\colon \psi_-(Z)\to Q_i, \qquad \eta \mapsto u
\end{equation*}
to~\cref{eq:heateq} is a linear and bounded map. Furthermore, $R_{i,0}$ is also a right-inverse to $T_{i,0}\colon\phi_-(W_i)\to  \psi_-(Z)$. Next, we introduce the trace operator from the evolving domain $\Omega_i$ to the evolving interface~$\Gamma$ by
\begin{equation*}
T_i=\psi\, T_{i,0}\,\phi_-\colon W_i\rightarrow Z,
\end{equation*}
which again becomes linear, bounded, and surjective. This trace operator has a bounded right-inverse given by 
\begin{equation*}
R_i=\phi\, R_{i,0}\,\psi_-\colon Z\rightarrow W_i.
\end{equation*}

\begin{lemma}\label{lem:nonhomexist}
If~\cref{ass:phi,ass:eqdata,ass:tech} hold, then for every $\eta\in Z$ and $g\in L^2_{V_i^*}(\R_+)$ there exists a solution $u\in W_i$ to 
\begin{equation}\label{eq:weak2}
a_i(u,v)=\langle g,v\rangle\quad\text{for all }v\in U_i^0,
\end{equation}
such that $T_iu=\eta$, $\phi_{-}u\in H^1_0\bigl(\mathbb{R}_+;H^{-1}(\Omega_i(0))\bigr)$, and  
\begin{equation}\label{eq:weaksoletabound}
\|u\|_{W_i}\leq C (\|g\|_{L^2_{H^{-1}(\Omega_i)}(\mathbb{R}_+)}+\|\eta\|_Z).
\end{equation}
\end{lemma}
\begin{proof}
Let $\eta$ be an arbitrary, but fixed, element in $Z$. First we prove that $a_i(u_\eta,\cdot)$ can be extended to an element in $L^2_{H^{-1}(\Omega_i)}(\mathbb{R}_+)$, where $u_\eta=R_i\eta\in W_i$. To this end, consider $v\in \phi_-(U_i^0)$. As $|J_t|v\in \phi_-(U_i^0)$ by~\cref{lem:Jvtest} and $\phi_- u_\eta$ solves~\cref{eq:heateq}, we have
\begin{equation*}
\begin{aligned}
d_i(\phi_- u_\eta, v)&=-\int_{\mathbb{R}_+}\int_{\Omega_i(0)} \phi_{-t}u_\eta\,\partial_tv\,|J_t|\,\mathrm{d}x_0\mathrm{d}t\\ 
                    &=-\int_{\mathbb{R}_+}\int_{\Omega_i(0)} \phi_{-t}u_\eta\bigl(\partial_t(|J_t|v)-\partial_t(|J_t|)v\bigr)\mathrm{d}x_0\mathrm{d}t\\
&=\int_{\mathbb{R}_+}\int_{\Omega_i(0)} -\nabla (\phi_{-t}u_\eta)\cdot\nabla(|J_t|v)+\partial_t(|J_t|)\phi_{-t}u_\eta v\bigr)\mathrm{d}x_0\mathrm{d}t.
\end{aligned}
\end{equation*}
The assumptions on $J$  together with the bound  $\|u_\eta\|_{L^2_{H^1(\Omega_i)}(\mathbb{R}_+)}\leq \|R_i\eta\|_{W_i}\leq C\|\eta\|_Z$ yield that
\begin{equation*}
\begin{aligned}
|a_i(u_\eta, v)|&\leq|d_i(\phi_-u_\eta,\phi_-v)|+|c_i(u_\eta, v)|\\
&\leq\|J\|_{L^\infty(\mathbb{R}_+; W^{1,\infty}(B_r))}\|u_\eta\|_{L^2_{H^1(\Omega_i)}(\mathbb{R}_+)}\|v\|_{L^2_{H^1(\Omega_i)}(\mathbb{R}_+)}\\
&\quad+\|J\|_{W^{1,\infty}(\mathbb{R}_+; L^\infty(B_r))}\|u_\eta\|_{L^2_{L^2(\Omega_i)}(\mathbb{R}_+)}\|v\|_{L^2_{L^2(\Omega_i)}(\mathbb{R}_+)}\\
& \quad+C\|u_\eta\|_{L^2_{H^1(\Omega_i)}(\mathbb{R}_+)}\|v\|_{L^2_{H^1(\Omega_i)}(\mathbb{R}_+)}\\
&\leq C\|\eta\|_Z\|v\|_{L^2_{H^1(\Omega_i)}(\mathbb{R}_+)}
\end{aligned}
\end{equation*}
for every $v\in U_i^0$. As $\mathcal{D}_0\subset U_i^0$ is dense in $L^2_{H_0^1(\Omega_i)}(\mathbb{R}_+)$, by~\cref{lem:density1}, the above bound implies that $a_i(u_\eta,\cdot)$ can be extended to a $L^2_{H^{-1}(\Omega_i)}(\mathbb{R}_+)$-functional.

Secondly, we construct a solution to~\cref{eq:weak2} with the inhomogeneous boundary data~$\eta$. By~\cref{lem:homexist} there exists a solution $u_0\in W_i^0$ to the equation 
\begin{equation*}
a_i(u_0,v)=\langle g,v\rangle-a_i(u_\eta,v)\quad\text{for all }v\in U_i^0
\end{equation*}
such that $\phi_-u_0\in H^1_0\bigl(\mathbb{R}_+;H^{-1}(\Omega_i(0))\bigr)$ and 
\begin{equation*}
\begin{aligned}
\|u_0\|_{W_i} &\leq C \bigl(\|\phi_-u_0\|^2_{L^2(\mathbb{R}_+;H^1(\Omega_i(0))}+\|\partial_t(\phi_-u_0)\|^2_{L^2(\mathbb{R}_+;H^{-1}(\Omega_i(0)))}\bigr)^{1/2}\\
&\leq C (\|g\|_{L^2_{H^{-1}(\Omega_i)}(\mathbb{R}_+)}+\|a_i(u_\eta,\cdot)\|_{L^2_{H^{-1}(\Omega_i)}(\mathbb{R}_+)})\\
&\leq C (\|g\|_{L^2_{H^{-1}(\Omega_i)}(\mathbb{R}_+)}+\|\eta\|_Z).
\end{aligned}
\end{equation*}
Hence, $u=u_0+u_\eta\in W_i$ solves~\cref{eq:weak2} with $T_iu=0+\eta$, 
\begin{equation*}
\begin{gathered}
\phi_- u =\phi_- u_0+R_{i,0}\,\psi_-\eta\in H^1_0\bigl(\mathbb{R}_+;H^{-1}(\Omega_i(0))\bigr),\quad{and}\\[5pt]
\|u\|_{W_i} \leq \|u_0\|_{W_i} +\|R_i\eta\|_{W_i} \leq C (\|g\|_{L^2_{H^{-1}(\Omega_i)}(\mathbb{R}_+)}+\|\eta\|_Z)+ C \|\eta\|_Z,
\end{gathered}
\end{equation*}
which concludes the proof.
\end{proof}

Introduce the space
\begin{equation*}
\tilde{\mathcal{D}}_0=\bigl\{u\in L^2_{L^2(\Omega_i)}(\mathbb{R}_+): \phi_-u\in C^{\infty}_0\bigl(\mathbb{R}_+\times B_r\bigr)\big\}.
\end{equation*}
\begin{lemma}\label{lem:density2}
If~\cref{ass:phi} holds, then $\tilde{\mathcal{D}}_0$ is dense in both $H^{1/2}_{L^2(\Omega_i)}(\mathbb{R}_+)$ and $H^{1/2}_{(0,\cdot)\, L^2(\Omega_i)}(\mathbb{R}_+)$. 
\end{lemma}
\begin{proof}
By \cite[Theorem 11.1]{lionsmagenes1}, one has that $C^\infty_0(\mathbb{R}_+)$ is dense in both $L^2(\mathbb{R}_+)$ and $H^{1/2}(\mathbb{R}_+)$.
Then the interpolation 
\begin{equation*}
H^{1/2}_{(0,\cdot)}(\mathbb{R}_+)\cong [H^1_0(\mathbb{R}_+), L^2(\mathbb{R}_+)]_{\frac12};
\end{equation*}
see~\cite[Theorem 11.7, Remark 2.6] {lionsmagenes1}, implies that $H^1_0(\mathbb{R}_+)$ is dense in $H_{(0,\cdot)}^{1/2}(\mathbb{R}_+)$. By definition of $H^1_0(\mathbb{R}_+)$ and~\cite[Proposition 2.3]{lionsmagenes1}, one obtains that $C^\infty_0(\mathbb{R}_+)$ is also dense in $H_{(0,\cdot)}^{1/2}(\mathbb{R}_+)$. The density of $\tilde{\mathcal{D}}_{0}$ in $H^{1/2}_{L^2(\Omega_i)}(\mathbb{R}_+)$ and $H^{1/2}_{(0,\cdot)\, L^2(\Omega_i)}(\mathbb{R}_+)$ then both follow by the very same tensor argument as in~\cref{lem:density1}.
\end{proof}

\begin{lemma}\label{lem:Jmult}
If~\cref{ass:phi,ass:eqdata,ass:tech} hold, then the map $v\mapsto |J|v$ is continuous on $H_{(0,\cdot)}^{1/2}\bigl(\mathbb{R}_+;L^2(\Omega_i(0))\bigr)$.
\end{lemma}
\begin{proof}
Under the assumptions it is clear that $|J|v\in L^2\bigl(\mathbb{R}_+;L^2(\Omega_i(0))\bigr)$ for every $v\in H_{(0,\cdot)}^{1/2}\bigl(\mathbb{R}_+;L^2(\Omega_i(0))\bigr)$, and
\begin{equation*}
\begin{aligned}
&\||J|v\|^2_{H_{(0,\cdot)}^{1/2}(\mathbb{R}_+;L^2(\Omega_i(0)))}
=\|e_{\mathbb{R}}(|J|v)\|^2_{H^{1/2}(\mathbb{R};L^2(\Omega_i(0)))}\\
&=\int_{\mathbb{R}_+}\int_{\mathbb{R}_+}\frac{ \||J_\tau|v(\tau)- |J_t|v(t)\|^2_{L^2(\Omega_i(0))} }{(\tau-t)^2}\,\mathrm{d}\tau\mathrm{d}t\\
&\quad+2\int_{\mathbb{R}_+}\frac{ \||J_t|v(t)\|^2_{L^2(\Omega_i(0))} }{t}\,\mathrm{d}t+\||J|v\|^2_{L^2(\mathbb{R}_+;L^2(\Omega_i(0)))}=K_1+K_2+K_3.
\end{aligned}
\end{equation*}
The integrals $K_2, K_3$ are trivially bounded by $\|v\|_{H_{(0,\cdot)}^{1/2}(\mathbb{R}_+;L^2(\Omega_i(0)))}$ as $|J|$ is an element in $L^\infty\bigl(\mathbb{R}_+; L^\infty(\Omega_i(0))\bigr)$. 

Next, denote the integrand of $K_1$ by $p(\tau,t)$. Then, as $p(\tau,t)=p(t,\tau)$, one has the equality 
\begin{equation*}
\begin{aligned}
K_1&=2\int_{\mathbb{R}_+}\int_0^tp(\tau,t)\,\mathrm{d}\tau\mathrm{d}t
=2\left(\int_1^\infty\int_0^{t-1}+\int_1^\infty\int_{t-1}^t+\int_0^1\int_0^{t}\right)p(\tau,t)\,\mathrm{d}\tau\mathrm{d}t\\
&=I_1+I_2+I_3.
\end{aligned}
\end{equation*}
On $I_1$'s domain of integration one has that $(\tau-t)^{-2}\leq 1$, i.e.,
\begin{equation*}
I_1\leq C\|J\|^2_{L^\infty(\mathbb{R}_+; L^\infty(\Omega_i(0)))}\|v\|^2_{L^2(\mathbb{R}_+;L^2(\Omega_i(0)))}.
\end{equation*}
The assumption $|J|\in W^{1,\infty}\bigl(\mathbb{R}_+; L^\infty(\Omega_i(0))\bigr)$ implies that
\begin{equation*}
\begin{aligned}
\||J_\tau|v(\tau) - |J_t|v(t)\|_{L^2(\Omega_i(0))}&=\|(|J_\tau|-|J_t|)v(t)- |J_\tau|(v(\tau)-v(t))\|_{L^2(\Omega_i(0))}\\
&\leq \|J\|_{W^{1,\infty}(\mathbb{R}_+; L^\infty(\Omega_i(0)))}(\tau-t)\|v(t)\|_{L^2(\Omega_i(0))}\\
&\quad+\|J\|_{L^\infty(\mathbb{R}_+; L^\infty(\Omega_i(0)))}\|v(\tau)-v(t)\|_{L^2(\Omega_i(0))}.
\end{aligned}
\end{equation*}
for a.e.\ $\tau,t\in \mathbb{R}_+$. Employing the above bound to the integrand $p$ yields
\begin{equation*}
\begin{aligned}
I_2&\leq C\|J\|^2_{W^{1,\infty}(\mathbb{R}_+; L^\infty(\Omega_i(0)))}\int_1^\infty\left(\int_{t-1}^t1\,\mathrm{d}\tau\right)\|v(t)\|^2_{L^2(\Omega_i(0))}\,\mathrm{d}t\\
&\quad+C\|J\|^2_{L^\infty(\mathbb{R}_+; L^\infty(\Omega_i(0)))}\int_1^\infty\int_{t-1}^t\frac{ \|v(\tau)-v(t)\|^2_{L^2(\Omega_i(0))}}{(\tau-t)^2} \,\mathrm{d}\tau\mathrm{d}t\\
&\leq C\|v\|^2_{H^{1/2}(\mathbb{R}_+;L^2(\Omega_i(0)))}.
\end{aligned}
\end{equation*}
The final integral $I_3$ can be bounded in the same fashion as $I_2$. 
\end{proof}

\begin{lemma}\label{lem:dextend}
If~\cref{ass:phi,ass:eqdata,ass:tech} hold, then the bilinear form 
\begin{equation*}
d_i\bigl(\phi_-(\cdot),\phi_-(\cdot)\bigr)\colon H^{1/2}_{(0,\cdot)\, L^2(\Omega_i)}(\mathbb{R}_+)\times H^1_{L^2(\Omega_i)}(\mathbb{R}_+)\to\mathbb{R}
\end{equation*} can be continuously extended to $H^{1/2}_{(0,\cdot)\, L^2(\Omega_i)}(\mathbb{R}_+)\times H^{1/2}_{L^2(\Omega_i)}(\mathbb{R}_+)$, and  the extension satisfies 
\begin{equation}\label{eq:dmon2}
d_i(\phi_{-}u,\phi_{-}u)\geq \int_{\mathbb{R}_+}\int_{\Omega_i(t)}(1/2\nabla \cdot \mathbf{w})u^2\, \mathrm{d}x_t\mathrm{d}t
\end{equation}
for all $u\in H^{1/2}_{(0,\cdot)\, L^2(\Omega_i)}(\mathbb{R}_+)$.
\end{lemma}
\begin{proof}
First observe the characterizations (cf.~\cite[Section~2]{EEEH25}) 
\begin{equation*}
\begin{aligned}
H^{1/2}\bigl(\mathbb{R}_+; L^2(\Omega_i(0))\bigr)&=\{u\in L^2(\mathbb{R}_+;  L^2(\Omega_i(0))): e_{\text{even}} u\in H^{1/2}\bigl(\mathbb{R}; L^2(\Omega_i(0))\bigr)\},\\
H^{1/2}_{(0,\cdot)}\bigl(\mathbb{R}_+; L^2(\Omega_i(0))\bigr)&=\{u\in L^2(\mathbb{R}_+;  L^2(\Omega_i(0))): e_\mathbb{R} u\in H^{1/2}\bigl(\mathbb{R}; L^2(\Omega_i(0))\bigr)\},
\end{aligned}	   
\end{equation*}
together with the equivalent norms
\begin{equation*}
\begin{aligned}	
\|u\|_{H^{1/2}(\mathbb{R}_+; L^2(\Omega_i(0)))}&=\|e_{\text{even}}u\|_{H^{1/2}(\mathbb{R}; L^2(\Omega_i(0))},\\
\|u\|_{H^{1/2}_{(0,\cdot)}(\mathbb{R}_+; L^2(\Omega_i(0)))}&=\|e_\mathbb{R} u\|_{H^{1/2}(\mathbb{R}; L^2(\Omega_i(0)))}.
\end{aligned}	
\end{equation*}
Here, the operators $e_\text{even},e_\mathbb{R}:L^2\bigl(\mathbb{R}_+;  L^2(\Omega_i(0))\bigr)\to L^2\bigl(\mathbb{R};  L^2(\Omega_i(0))\bigr)$ denote the even extension and the extension by zero, respectively, in the temporal direction.

The bilinear form $d_i(\cdot,\cdot):H^{1/2}_{(0,\cdot)}\bigl(\mathbb{R}_+;L^2(\Omega_i(0))\bigr)\times\phi_-(\tilde{\mathcal{D}}_{0})\to \mathbb{R}$ then satisfies the bound 
\begin{equation}\label{eq:dextend}
\begin{aligned}	
|d_i(u,v)|&=\left|\int_{\mathbb{R}_+}\int_{\Omega_i(0)}u \partial_tv |J_t|\,\mathrm{d}x_0\mathrm{d}t\right| =\left|\int_\mathbb{R}\int_{\Omega_i(0)} e_\mathbb{R} (|J_t| u)\partial_t(e_{\text{even}}v)\,\mathrm{d}x_0\mathrm{d}t\right|\\
            &\leq C\|e_\mathbb{R} (|J_t| u)\|_{H^{1/2}(\mathbb{R}, L^2(\Omega_i(0)))}\|e_{\text{even}}v\|_{H^{1/2}(\mathbb{R}, L^2(\Omega_i(0)))}\\
            &\leq C\|u\|_{H_{(0,\cdot)}^{1/2}(\mathbb{R}_+; L^2(\Omega_i(0)))}\|v\|_{H^{1/2}(\mathbb{R}_+; L^2(\Omega_i(0)))},
\end{aligned}	
\end{equation}
where the first inequality follows as in~\cite[Section~5]{EEEH24} and the second one holds due to~\cref{lem:Jmult}. 

By~\cref{lem:density2}, the set $\phi_-(\tilde{\mathcal{D}}_{0})$ is dense in $H^{1/2}\bigl(\mathbb{R}_+;L^2(\Omega_i(0))\bigr)$. Hence,~\cref{eq:dextend} yields that~$d_i$ can be continuously extended to $H^{1/2}_{(0,\cdot)}\bigl(\mathbb{R}_+;L^2(\Omega_i(0))\bigr)\times H^{1/2}\bigl(\mathbb{R}_+;L^2(\Omega_i(0))\bigr)$, and 
\begin{equation*}
d_i\bigl(\phi_-(\cdot),\phi_-(\cdot)\bigr)\colon H^{1/2}_{(0,\cdot)\, L^2(\Omega_i)}(\mathbb{R}_+)\times H^{1/2}_{L^2(\Omega_i)}(\mathbb{R}_+)\to\mathbb{R}
\end{equation*}
is then also a well defined, bounded bilinear form.

Finally, applying~\cref{lem:dmon} to an element $u\in \tilde{\mathcal{D}}_{0}\subset H^1_{L^2(\Omega_i)}(\mathbb{R}_+)$ and observing that~$\tilde{\mathcal{D}}_{0}$ is dense in~$H^{1/2}_{(0,\cdot)\, L^2(\Omega_i)}(\mathbb{R}_+)$, see~\cref{lem:density2}, yields the lower bound~\cref{eq:dmon2}.
\end{proof}

\begin{lemma}\label{lem:density3}
If~\cref{ass:phi} holds, then $U_i$ and $U_i^0$ are dense in~$\tilde{W}_i$ and~$\tilde{W}_i^0$, respectively. 
\end{lemma}
\begin{proof}
We first consider the density of $U_i$. Introduce the mollifier $\varphi\in C^\infty_0(\mathbb{R})$ with the property $\int_{\mathbb{R}}\varphi(t)\mathrm{d}t=1$ and let 
$\varphi_{\varepsilon}(t)=\varepsilon^{-1}\varphi(\varepsilon^{-1}t)$ for $\varepsilon>0$. For every 
\begin{equation*}
v\in H^{1/2}\bigl(\mathbb{R};L^2(\Omega_i(0))\bigr)\cap L^2\bigl(\mathbb{R};V_i(0)\bigr)
\end{equation*}
the convolution $v_{\varepsilon}=\varphi_{\varepsilon}*v$ is an element in $L^2\bigl(\mathbb{R};V_i(0)\bigr)$ and $\{v_{\varepsilon}\}$ converges to $v$ in the same space as $\varepsilon$ tends to~$0^+$, see, e.g.,~\cite[Lemma~1.2.30 and~Proposition~1.2.32]{hytonen}.

Recall the vector-valued Fourier transform $\mathcal{F}$ on $L^2\bigl(\mathbb{R};L^2(\Omega_i(0))\bigr)_\mathbb{C}$, see~\cite[Sections~2.4]{hytonen}. The Fourier characterizations of convolutions and derivatives then imply that $\mathcal{F}(v_\varepsilon)=\mathcal{F}(\varphi_\varepsilon)\mathcal{F}(v)$ and $\partial_tv_\varepsilon=(\partial_t\varphi_{\varepsilon})*v$. The latter implies that 
\begin{equation*}
v_\varepsilon\in H^1\bigl(\mathbb{R};L^2(\Omega_i(0))\bigr)\cap L^2\bigl(\mathbb{R};V_i(0)\bigr) 
\end{equation*}
for all $\epsilon>0$. Furthermore, the Fourier characterization of $H^{1/2}(\mathbb{R})$, see~\cite[Lemma~16.3]{tartar}, yields that  
\begin{equation*}
v\mapsto \bigl(\|\sqrt{\mathrm{i}(\cdot)}\mathcal{F}v\|^2_{L^2(\mathbb{R};L^2(\Omega_i(0)))_\mathbb{C}}+\|v\|^2_{L^2(\mathbb{R};L^2(\Omega_i(0)))}\bigr)^{1/2}
\end{equation*}
is an equivalent norm on $H^{1/2}\bigl(\mathbb{R};L^2(\Omega_i(0))\bigr)$. Due to the temporal $H^{1/2}$-regularity of~$v$, one has that $w=\mathcal{F}^{-1}\sqrt{\mathrm{i}(\cdot)}\mathcal{F}v\in L^2\bigl(\mathbb{R};L^2(\Omega_i(0))\bigr)$ and 
\begin{equation*}
\begin{aligned}
\|v-v_{\varepsilon} \|^2_{H^{1/2}(\mathbb{R};L^2(\Omega_i(0)))}&\leq C\| (1-\mathcal{F}\varphi_\varepsilon)\sqrt{\mathrm{i}(\cdot)} \mathcal{F}v\|^2_{L^2(\mathbb{R};L^2(\Omega_i(0)))_\mathbb{C}}\\
&\qquad+\|v-v_{\varepsilon}\|^2_{L^2(\mathbb{R};L^2(\Omega_i(0)))}\\
&\leq C\|w-\varphi_{\varepsilon}*w\|^2_{L^2(\mathbb{R};L^2(\Omega_i(0)))}+\|v-v_{\varepsilon}\|^2_{L^2(\mathbb{R};L^2(\Omega_i(0)))}.
\end{aligned}
\end{equation*}
The above bound together with~\cite[Proposition~1.2.32]{hytonen} implies that $v_{\varepsilon}$ also converges to $v$ in~$H^{1/2}\bigl(\mathbb{R};L^2(\Omega_i(0))\bigr)$. 

Next, let $u\in \phi_-(\tilde{W}_i)$ and observe that the even extension in time 
\begin{equation*}
e_{\mathrm{even}}\colon \phi_-(\tilde{W}_i)\to H^{1/2}\bigl(\mathbb{R};L^2(\Omega_i(0))\bigr)\cap L^2\bigl(\mathbb{R};V_i(0)\bigr) 
\end{equation*}
is a well defined map. Furthermore, the restriction $r_{\mathbb{R}_+}v=\restr{v}{\mathbb{R}_+\times \Omega_i(0)}$ is a bounded left-inverse to $e_{\mathrm{even}}$. Set $u_\varepsilon=r_{\mathbb{R}_+}(\varphi_{\varepsilon}*e_{\mathrm{even}}u)\in \phi_-(U_i)$. The previous argumentation on $\mathbb{R}$ then gives us the limit
\begin{equation*}
\begin{aligned}
\|u-u_\varepsilon\|_{\phi_-(\tilde{W}_i)}&=\|r_{\mathbb{R}_+}(e_{\mathrm{even}}u-\varphi_{\varepsilon}*e_{\mathrm{even}}u)\|_{\phi_-(\tilde{W}_i)}\\
&\leq C\|e_{\mathrm{even}}u-\varphi_{\varepsilon}*e_{\mathrm{even}}u\|_{H^{1/2}(\mathbb{R};L^2(\Omega_i(0)))\cap L^2(\mathbb{R};V_i(0))}\to 0
\end{aligned}
\end{equation*}
as $\varepsilon$ tends to~$0^+$. Hence, $\phi_-(U_i)$ is dense in $\phi_-(\tilde{W}_i)$. The proof for $U_i$ is then completed by recalling that $\phi\colon H^{1/2}\bigl(\mathbb{R};L^2(\Omega_i(0))\bigr)\cap L^2\bigl(\mathbb{R};V_i(0)\bigr)\to \tilde{W}_i$ is isomorphic with an equivalence of norms. 

The same argument holds for $U_i^0$ simply by replacing $V_i$ with $H^1_0(\Omega_i)$.
\end{proof}

\begin{corollary}\label{cor:aextend}
If~\cref{ass:phi,ass:eqdata,ass:tech} hold, then the bilinear form\linebreak $a_i\colon W_i\times U_i\to \mathbb{R}$
can be continuously extended to $W_i\times\tilde{W}_i$, and the extension satisfies 
\begin{equation}\label{eq:amon2}
a_i(u,u)\geq c\|u\|^2_{L^2_{H^1(\Omega_i)}(\mathbb{R}_+)}\quad\text{for all }u\in W_i.
\end{equation}
\end{corollary}
\begin{proof}
By~\cref{lem:dextend} one has the bound
\begin{equation*}
\begin{aligned}
|a_i(u,v)| &\leq |d_i\bigl(\phi_-u,\phi_-v\bigr)|+|c_i(u,v)|\\
            &\leq C\bigl(\|u\|_{H^{1/2}_{(0,\cdot)\, L^2(\Omega_i)}(\mathbb{R}_+)}\|u\|_{H^{1/2}_{ L^2(\Omega_i)}(\mathbb{R}_+)}
                +\|u\|_{L^2_{H^1(\Omega_i)}(\mathbb{R}_+)}\|v\|_{L^2_{H^1(\Omega_i)}(\mathbb{R}_+)}\bigr)\\
             &\leq C\|u\|_{W_i}\|v\|_{\tilde{W}_i}
\end{aligned}
\end{equation*}
for all $u\in W_i$ and $v\in U_i$. As $U_i$ is dense in $\tilde{W}_i$, via~\cref{lem:density3}, the above bound implies that~$a_i$ can be continuously extended to $W_i\times\tilde{W}_i$. The lower bound~\cref{eq:amon2} follows directly by combining \cref{eq:dmon2} with \cref{ass:eqdata}.
\end{proof}

This temporal $H^{1/2}$-framework now gives a \emph{unique} solution to the weak parabolic equation on $\Omega_i$ with inhomogeneous boundary conditions.

\begin{corollary}\label{cor:weakHonehalf}
If~\cref{ass:phi,ass:eqdata,ass:tech} hold, then for every $\eta\in Z$ and $g\in L^2_{V_i^*}(\R_+)$ there exists a unique solution $u\in W_i$ to the equation 
\begin{equation}\label{eq:weakHonehalf_i}
a_i(u,v)=\langle g,v\rangle\quad\text{for all }v\in \tilde{W}_i^0
\end{equation}
such that $T_iu=\eta$, $\phi_-u\in H^1_0\bigl(\mathbb{R}_+;H^{-1}(\Omega_i(0))\bigr)$, and $u$ satisfies the bound~\cref{eq:weaksoletabound}.
\end{corollary}
\begin{proof}
According to~\cref{lem:nonhomexist}, there exists a solution $u\in W_i$ to~\cref{eq:weak2}, with the desired properties. The density of $U_i^0$ in $\tilde{W}_i^0$ together with the extension of~$a_i$, see~\cref{lem:density3,cor:aextend}, then implies that~$u$ is also a solution to~\cref{eq:weakHonehalf_i}. The uniqueness of the solution follows directly by~\cref{eq:amon2}.
\end{proof}

The same $H^{1/2}$-extension that has lead up to \cref{cor:aextend,cor:weakHonehalf} trivially gives that the weak problem on $\Omega$ with 
homogeneous boundary conditions, i.e., 
\begin{equation}\label{eq:weakHonehalf}
a(u,v)=\langle f,v\rangle\quad\text{for all }v\in \tilde{W},
\end{equation}
has a unique solution $u\in W$ with $\phi_- u\in H_0^1\bigr(\mathbb{R}_+;H^{-1}(\Omega(0))\bigr)$.

Letting $g=0$ in~\cref{cor:weakHonehalf} yields the bounded linear solution operators 
\begin{equation}\label{eq:Fi}
F_i\colon Z\to W_i, \qquad \eta\mapsto u 
\end{equation}
such that $T_iF_i\eta=\eta$, $\phi_-F_i\eta\in Q_i$ and $u=F_i\eta$ solves~\cref{eq:weakHonehalf_i} with $g=0$. 

Moreover, setting $\eta=0$ yields the bounded linear solution operators 
\begin{equation*}
G_i\colon L^2_{H^{-1}(\Omega_i)}(\R_+)\to W_i^0, \qquad g\mapsto u 
\end{equation*}
such that $u=G_ig$ solves~\cref{eq:weakHonehalf_i}. Here we use the fact that every $g\in L^2_{H^{-1}(\Omega_i)}(\R_+)$ can be interpreted as an element in $L^2_{V_i^*}(\mathbb{R}_+)$ by restricting $g$ to $L^2_{V_i}(\mathbb{R}_+)$. (This restriction is not an injective map.)

\section{Transmission problems on evolving domains}\label{sec:trans}
In this section we analyze the transmission problem on the evolving domain decomposition~\cref{eq:trans}. The weak formulation of~\cref{eq:trans} is to find $(u_1, u_2)\in W_1\times W_2$ such that
\begin{equation}\label{eq:weaktran}
	\left\{\begin{aligned}
	     a_i(u_i, v_i)&=\langle f_i, v_i\rangle & & \text{for all } v_i\in \tilde{W}_i^0,\, i=1,2,\\
	     T_1u_1&=T_2u_2, & &\\
	     \textstyle\sum_{i=1}^2 a_i(u_i, F_i\mu)-\langle f_i, F_i\mu\rangle&=0 & &\text{for all }\mu\in Z. 
	\end{aligned}\right.
\end{equation}
We introduce the spatial restriction operators
\begin{equation*}
    q_{i,0}\colon u\mapsto\restr{u}{\R_+\times\Omega_i(0)}\quad\text{and}\quad q_i=\phi q_{i,0}\phi_{-},
\end{equation*}
where the maps $q_i\colon W\to W_i$ and $q_i\colon\tilde{W}\to \tilde{W}_i$ both become well defined and continuous. With the derived setting we are now able to ``cut'' and ``glue together'' functions without losing spatial or temporal regularity. 
\begin{lemma}\label{lemma:paste}
 Suppose that~\cref{ass:phi} holds. If $u\in W$, then 
 \begin{equation*}
 (u_1, u_2)=(q_1u, q_2u)\in W_1\times W_2
 \end{equation*}
and $T_1u_1=T_2u_2$. Conversely, if $(u_1, u_2)\in W_1\times W_2$ and $T_1u_1=T_2u_2$, then $u=\phi v$ with 
 \begin{equation*}
v=\{\phi_-u_1 \text{ on } \R_+\times\Omega_1(0), \phi_-u_2\text{ on } \R_+\times\Omega_2(0)\}
 \end{equation*}
satisfies $u\in W$. Moreover, the same holds for $u\in \tilde{W}$ and $(u_1, u_2)\in \tilde{W}_1\times \tilde{W}_2$.
\end{lemma}
\begin{proof}
    First, suppose that $u\in W$. Then, as $\phi_-$ is an isomorphism on intersection spaces, we have
    \begin{equation*}
        \phi_-u\in H_{(0,\cdot)}^{1/2}\bigl(\R_+; L^2(\Omega(0))\bigr)\cap L^2\bigl(\R_+; H^1_0(\Omega(0))\bigr).
    \end{equation*}
    It follows from~\cite[Lemma 5]{EEEH25} that
    \begin{equation*}
        q_{i,0}(\phi_-u)\in H_{(0,\cdot)}^{1/2}(\R_+; L^2(\Omega_i(0)))\cap L^2\bigl(\R_+; V_i(0)\bigr)
    \end{equation*}
     and $T_{1,0}\restr{(\phi_-u)}{\R_+\times\Omega_1(0)}=T_{2,0}\restr{(\phi_-u)}{\R_+\times\Omega_2(0)}$. Applying the isomorphism $\phi$ yields that $u_i=q_iu\in W_i$ and the definition of $T_i$ gives $T_1u_1=T_2u_2$. 
    
Conversely, suppose that $u_i\in W_i$ and $T_1u_1=T_2u_2$. Then we have
    \begin{equation*}
        \phi_-u_i\in H_{(0,\cdot)}^{1/2}(\R_+; L^2(\Omega_i(0)))\cap L^2\bigl(\R_+; V_i(0)\bigr)
    \end{equation*}
    and by the definition of $T_i$ we have $T_{1, 0}\phi_-u_1=T_{2, 0}\phi_- u_2$. Therefore, it follows from~\cite[Lemma 5]{EEEH25} that
    \begin{equation*}
    \begin{aligned}
        v&=\{\phi_-u_1 \text{ on } \R_+\times\Omega_1(0), \phi_-u_2\text{ on } \R_+\times\Omega_2(0)\}\\
        &\in H_{(0,\cdot)}^{1/2}(\R_+; L^2(\Omega(0)))\cap L^2\bigl(\R_+; H^1_0(\Omega(0))\bigr).
    \end{aligned}
    \end{equation*}
Thus, we have $u=\phi v\in W$. The argument for $\tilde{W}$ and $\tilde{W}_i$ is the same since it also holds on the reference cylinder $\R_+\times\Omega(0)$ according to~\cite[Lemma 5]{EEEH25}.
\end{proof}

\begin{lemma}\label{lemma:asplit}
If~\cref{ass:phi,ass:eqdata,ass:tech} hold, then 
    \begin{equation*}
        a(u, v) =\sum_{i=1}^2a_i(q_iu, q_iv)
    \end{equation*}
    for all $u\in W$ and $v\in \tilde{W}$.
\end{lemma}
\begin{proof}
First, let $u\in H^{1/2}_{(0,\cdot)}\bigl(\mathbb{R}_+;L^2(\Omega)\bigr)$ and consider the restriction operator
\begin{equation*}
q_{i,0}\colon H^{1/2}\bigl(\mathbb{R}_+; L^2(\Omega(0))\bigr)\to H^{1/2}(\mathbb{R}_+; L^2(\Omega_i(0))),
\end{equation*} 
which is continuous. The same holds if $H^{1/2}$ is replaced by $H_{(0,\cdot)}^{1/2}$. For $v\in\phi_-(\tilde{\mathcal{D}}_{0})$ we have $q_{i,0}(\partial_t v)=\partial_t(q_{i,0}v)$ and 
\begin{equation*}
\begin{aligned}
 d(u, v)&=-\int_{\R_+}\int_{\Omega(0)}u\partial_t v |J_t|\,\mathrm{d}x_0\mathrm{d}t=\sum_{i=1}^2-\int_{\R_+}\int_{\Omega_i(0)} u\partial_t(q_{i,0} v) |J_t|\,\mathrm{d}x_0\mathrm{d}t\\
        &=\sum_{i=1}^2d_i\bigl(q_{i,0}u,q_{i,0} v\bigr).
 \end{aligned}
\end{equation*} 
By~\cref{lem:density2}, we have for every $v\in H^{1/2}_{L^2(\Omega)}(\mathbb{R}_+)$ that there exists a sequence $\{v_n\}\subset \tilde{\mathcal{D}}_{0}$ that converges to $v$. Then the continuity of $d$, $d_i$, $\phi_-$, and $q_{i,0}$ gives 
\begin{equation*}
\begin{aligned}
 d(\phi_- u, \phi_-v)&=\lim_{n\to\infty}d(\phi_- u, \phi_- v_n)=\lim_{n\to\infty}\sum_{i=1}^2d_i\bigl(q_{i,0}(\phi_- u),q_{i,0}(\phi_-v_n)\bigr)\\
 &= \sum_{i=1}^2d_i\bigl(q_{i,0} (\phi_- u),q_{i,0}(\phi_- v)\bigr).
 \end{aligned}
\end{equation*} 
Second, let $u\in L^2_{H^1_0(\Omega)}(\mathbb{R}_+)$. As $u(t)\in H^1_0\bigl(\Omega(t)\bigr)$ for a.e.\ $t\in \mathbb{R}_+$, one obtains 
\begin{equation*}
\begin{gathered}
\restr{\bigl(\nabla u (t)\bigr)}{\Omega_i(t)}=\nabla\bigl(\restr{u (t)}{\Omega_i(t)}\bigr)\quad\text{and}\\ 
\restr{u(t)}{\Omega_i(t)}=u(t)\circ \bigl(\restr{\Phi_{t}}{\Omega_i(0)}\circ\restr{\Phi_{-t}}{\Omega_i(t)}\bigr)=
\restr{\bigl(\phi_{-t}u(t)\bigr)}{\Omega_i(0)}\circ \restr{\Phi_{-t}}{\Omega_i(t)}=(q_i u)(t)
\end{gathered}
\end{equation*}
for a.e.\ $t\in \mathbb{R}_+$. Hence, for every $u,v\in L^2_{H^1_0(\Omega)}(\mathbb{R}_+)$ we have
\begin{equation*}
\begin{aligned}
 c(u, v)&=\int_{\R_+}\int_{\Omega(t)}\alpha\nabla u\cdot\nabla v+\beta uv \,\mathrm{d}x_t\mathrm{d}t
 =\sum_{i=1}^2\int_{\R_+}\int_{\Omega_i(t)}\restr{\bigl(\alpha\nabla u\cdot\nabla v+\beta uv\bigr)}{\Omega_i(t)}\,\mathrm{d}x_t\mathrm{d}t\\
 &=\sum_{i=1}^2\int_{\R_+}\int_{\Omega_i(t)}\alpha\nabla (q_iu)\cdot\nabla (q_iv)+\beta (q_iu)(q_iv) \,\mathrm{d}x_t\mathrm{d}t=\sum_{i=1}^2c_i\bigl(q_iu,q_iv\bigr).
 \end{aligned}
\end{equation*} 
Combining these results for $u\in W$ and $v\in \tilde{W}$ gives the sought-after equality. 
\end{proof}
\begin{theorem}\label{lemma:tranequiv}
    Suppose that~\cref{ass:phi,ass:eqdata,ass:tech} hold. The transmission problem is equivalent to the weak problem in the following way: If $u$ solves~\cref{eq:weakHonehalf}, then $(u_1, u_2)=(q_1u, q_2u)$ solves~\cref{eq:weaktran}. Conversely, if $(u_1, u_2)$ solves~\cref{eq:weaktran}, then $u=\phi v$ with 
    $ v=\{\phi_-u_1 \text{ on } \R_+\times\Omega_1, \phi_-u_2\text{ on } \R_+\times\Omega_2\}$ solves~\cref{eq:weakHonehalf}.
\end{theorem}
\begin{proof}
    Suppose that $u\in W$ solves~\cref{eq:weakHonehalf}. Then
    \begin{equation*}
        (u_1, u_2)=(q_1u, q_2u)\in W_1\times W_2
    \end{equation*}
    and $T_1u_1=T_2u_2$ by~\cref{lemma:paste}. Moreover, let $v_i=\phi w_i\in \tilde{W}_i^0$ with $w_i=\{\phi_-v_i \text{ on } \R_+\times\Omega_i, 0\text{ on } \R_+\times\Omega_{3-i}\}$ for $i=1,2$. It follows by~\cref{lemma:paste} and~\cref{eq:L2Videntities} that $w_i\in \tilde{W}$. Therefore, by~\cref{ass:eqdata,lemma:asplit} we have
    \begin{align*}
        a_i(u_i, v_i)&=a_i(q_iu, q_iw_i)=a_i(q_iu, q_iw_i)+a_{3-i}(q_{3-i}u, q_{3-i}w_i)\\
        &=a(u, w_i)=\langle f, w_i\rangle\\
        &=\langle f_i, q_iw_i\rangle+\langle f_{3-i}, q_{3-i}w_i\rangle
        =\langle f_i, v_i\rangle.
    \end{align*}
    Now let $\mu\in Z$ and define $v=\phi w$ by $w=\{\phi_-F_1\mu \text{ on } \R_+\times\Omega_1(0), \phi_-F_2\mu\text{ on } \R_+\times\Omega_2(0)\}$. Since $T_1F_1\mu=T_2F_2\mu$ we have that $v\in W$ according to~\cref{lemma:paste}. Thus~\cref{ass:eqdata,lemma:asplit} again yield
    \begin{align*}
        \sum_{i=1}^2 a_i(u_i, F_i\mu)=a(u, v)=\langle f, v\rangle=\sum_{i=1}^2 \langle f_i, F_i\mu\rangle
    \end{align*}
    and we have now shown that $u$ satisfies all three equations of~\cref{eq:weaktran}. 
    
    Conversely, suppose that $(u_1, u_2)$ solves~\cref{eq:weaktran} and define $u=\phi v$ with $v=\{\phi_-u_1 \text{ on } \R_+\times\Omega_1(0), \phi_-u_2\text{ on } \R_+\times\Omega_2(0)\}$. Since $T_1u_1=T_2u_2$ we have that $u\in W$ according to~\cref{lemma:paste}. Now let $v\in \tilde{W}$ and define $(v_1, v_2)=(q_1v, q_2v)$, which satisfies $v_i\in \tilde{W}_i$ and $T_1v_1=T_2v_2$ again by~\cref{lemma:paste}. If we define $\mu=T_iv_i$ then $v_i-F_i\mu\in \tilde{W}_i^0$ by~\cref{eq:L2Videntities}. Therefore, by~\cref{ass:eqdata,lemma:asplit}, we have
    \begin{align*}
        a(u, v)&=\sum_{i=1}^2 a_i(u_i, v_i)=\sum_{i=1}^2 a_i(u_i, v_i-F_i\mu)+a_i(u_i, F_i\mu)\\
        &=\sum_{i=1}^2 \langle f_i, v_i-F_i\mu\rangle+\langle f_i, F_i\mu\rangle=\sum_{i=1}^2 \langle f_i, v_i\rangle=\langle f, v\rangle,
    \end{align*}
    which means that $u$ solves~\cref{eq:weakHonehalf}.
\end{proof}

\section{Steklov--Poincaré operators and convergence of the Robin--Robin scheme on evolving domains}\label{sec:SP}

The Steklov--Poincaré operators $S_i,S\colon Z\to Z^*$ are defined as
\begin{displaymath}
    \langle S_i\eta,\mu\rangle=a_i(F_i\eta, F_i\mu)
\end{displaymath}
and $S=S_1+S_2$. Moreover, we define the functionals $\chi_i,\chi\in Z^*$ as
\begin{displaymath}
    \langle \chi_i,\mu\rangle=\langle f_i, F_i\mu\rangle-a_i(G_if_i, F_i\mu),
\end{displaymath}
and $\chi =\chi_1+\chi_2$.
\begin{remark}
    Although the bilinear forms $a_i$ have different trial and test spaces, the Steklov--Poincaré operators have the same trial and test spaces, i.e., $S_i\colon Z\to Z^*$. This is due to the fact that the spaces $W_i$ and $\Tilde{W}_i$ share the trace space $Z$, see~\cite[p. 507]{costabel90}.
\end{remark}
We will now prove the main properties of the Steklov--Poincaré operators, namely that they are bounded and coercive. We first have the following lemma, which is important for the coercivity of the Steklov--Poincaré operators.

\begin{lemma}\label{lem:Fibound}
If~\cref{ass:phi,ass:eqdata,ass:tech} hold then 
\begin{equation*}
\|F_i\eta\|_{W_i}\leq C \|F_i\eta\|_{L^2_{H^1(\Omega_i)}(\R_+)}
\end{equation*}
for every $\eta\in Z$.
\end{lemma}
\begin{proof}
By the definition of~$F_i$ in~\cref{eq:Fi} one has $\phi_-F_i\eta\in Q_i$, and the embedding~\cref{eq:keyinterpolate} then yields the bound 
\begin{equation*}
        \|F_i\eta\|_{W_i}\leq C\bigl(\|\phi_-F_i\eta\|^2_{L^2(\mathbb{R}_+;H^1(\Omega_i(0)))}+\|\partial_t\phi_-F_i\eta\|^2_{L^2(\mathbb{R}_+;H^{-1}(\Omega_i(0)))}\bigr)^{1/2}. 
\end{equation*}
Furthermore, as $u=F_i\eta$ is a solution to~\cref{eq:weakHonehalf_i}, the same calculations as in~\cref{eq:weaktimederivative1} together with the bound~\cref{eq:weaktimederivative2} for $g=0$  gives 
\begin{equation*}
\|\partial_t\phi_-F_i\eta\|_{L^2(\mathbb{R}_+;H^{-1}(\Omega_i(0)))}\leq C\|F_i\eta\|_{L^2_{H^1(\Omega_i)}(\R_+)}.
\end{equation*}
Combining these results gives the desired estimate.
\end{proof}
\begin{theorem}\label{lemma:sp}
Suppose that~\cref{ass:phi,ass:eqdata,ass:tech} hold. Then the Steklov--Poincaré operators $S_i$ are bounded and also coercive, i.e.,
\begin{equation}\label{eq:weakcoer}
    \langle S_i\eta, \eta\rangle\geq c\|\eta\|^2_Z\quad\text{for all }\eta\in Z.
\end{equation}
Moreover, a similar result holds for $S$.
\end{theorem}
\begin{proof}
By~\cref{cor:aextend} and the fact that $F_i\colon Z\to W_i$ is bounded we have
\begin{equation*}
    \bigl|\langle S_i\eta, \mu\rangle \bigr|=\bigl|a_i(F_i\eta, F_i\mu)\bigr|\leq C\|F_i\eta\|_{W_i}\|F_i\mu\|_{\tilde{W}_i}\leq C\|F_i\eta\|_{W_i}\|F_i\mu\|_{W_i}\leq C\|\eta\|_{Z}\|\mu\|_{Z}
\end{equation*}
    for all $\eta,\mu\in Z$, which shows that $S_i$ is bounded. It follows from~\cref{eq:amon2,lem:Fibound} that
    \begin{align*}
        \langle S_i\eta, \eta\rangle=a_i(F_i\eta, F_i\eta)\geq c\|F_i\eta\|_{L^2_{H^1(\Omega_i)}(\R_+)}^2\geq c\|F_i\eta\|_{W_i}^2\geq c\|T_iF_i\eta\|_Z^2= c\|\eta\|_Z^2
    \end{align*}
    for all $\eta\in Z$. Thus $S_i$ is coercive. The result for $S$ follows by summing the inequalities for~$S_i$.
\end{proof}
The properties of $S_i$ and $S$ immediately yield that the operators satisfy the assumptions of the Lax--Milgram theorem and therefore we have the following corollary.
\begin{corollary}\label{cor:spiso}
    The Steklov--Poincaré operators $S_i, S\colon Z\rightarrow Z^*$ are isomorphisms.
\end{corollary}
The Steklov--Poincaré equation is to find $\eta\in Z$ such that
\begin{equation}\label{eq:speq}
    S\eta=\chi.
\end{equation}
The following result is an immediate consequence of the definitions of the Steklov--Poincaré operators.
\begin{lemma}\label{lemma:tpspeq}
    Suppose that~\cref{ass:phi,ass:eqdata,ass:tech} hold. The Steklov--Poincaré equation is equivalent to the transmission problem in the following way: If $(u_1, u_2)$ solves~\cref{eq:weaktran} then $\eta=T_iu_i$ solves~\cref{eq:speq}. Conversely, if $\eta$ solves~\cref{eq:speq}, then $(F_1\eta+G_1f_1, F_2\eta+G_2f_2)$ solves~\cref{eq:weaktran}.
\end{lemma}
Any non-overlapping domain decomposition method can be formulated as an interface iteration to solve~\cref{eq:speq}, see, e.g.,~\cite{EEEH24,EEEH25,quarteroni}. We first consider the Robin--Robin method. For a parameter $s_0>0$ and an initial guess $\eta_2^0\in Z$, the interface iteration of the Robin--Robin method is given by finding $(\eta_1^n, \eta_2^n)\in Z\times Z$ for $n=1,2,\dots$ such that
\begin{equation}\label{eq:pr}
	\left\{\begin{aligned}
	     (s_0\mathcal{R}+S_1)\eta_1^n&=(s_0\mathcal{R}-S_2)\eta_2^{n-1}+\chi & & \\
	     (s_0\mathcal{R}+S_2)\eta_2^n&=(s_0\mathcal{R}-S_1)\eta_1^n+\chi. & &
	\end{aligned}\right.
\end{equation}
Here, $\mathcal{R}$ denotes the Riesz isomorphism defined as
\begin{equation*}
    \mathcal{R}\colon L^2_{L^2(\Gamma)}(\R_+)\to L^2_{L^2(\Gamma)}(\R_+)^*, \qquad \mu\mapsto (\mu, \cdot)_{L^2_{L^2(\Gamma)}(\R_+)}.
\end{equation*}
The following result follows as in the case of elliptic problems~\cite[Lemma 6.3]{EHEE22}. To be precise, we are referring to the weak formulation of~\cref{eq:robin}, which is of the same form as in the case of elliptic problems~\cite[(5.2)]{EHEE22}.
\begin{lemma}\label{lemma:rrpr}
Suppose that~\cref{ass:phi,ass:eqdata,ass:tech} hold. The Robin--Robin method and the Peaceman--Rachford iteration are equivalent in the following way: If $(u_1^n, u_2^n)_{n\geq 1}$ solves \cref{eq:robin} then $(\eta_1^n, \eta_2^n)_{n\geq 1}$, defined as $\eta_i^n=T_iu^n_i$, solves \cref{eq:pr} with $\eta_2^0=T_2u^0_2$. Conversely, if $(\eta_1^n, \eta_2^n)_{n\geq 1}$ solves \cref{eq:pr} then $(u_1^n, u_2^n)_{n\geq 1}$, defined as $u_i^n=F_i\eta_i^n+G_if_i$, solves \cref{eq:robin} with $u_2^0=F_2\eta_2^0+G_2f_2$.
\end{lemma}
The second method we discuss is the Dirichlet--Neumann method. For a method parameter $s_1>0$ and an initial guess $\eta^0\in Z$, the interface iteration corresponding to the Dirichlet--Neumann method is given by finding $\eta^n\in Z$ for $n=1,2,\dots$ such that
\begin{equation}\label{eq:dn}
    \eta^n=\eta^{n-1}+s_1S_2^{-1}(\chi-S\eta^{n-1}).
\end{equation}
Finally we consider the Neumann--Neumann method. For two method parameters $s_2,s_3>0$ and an initial guess $\eta^0\in Z$, the interface iteration corresponding to the Neumann--Neumann method is given by finding $(\eta^n,\lambda_1^n,\lambda_2^n)\in Z\times Z\times Z$ for $n=1,2,\dots$ such that
\begin{equation}\label{eq:nn}
	\left\{\begin{aligned}
	     S_i\lambda_i^n&=\chi-S\eta^{n-1} & & \text{for } i=1,2,\\
	     \eta^n&=\eta^{n-1}+s_2\lambda_1^n+s_3\lambda_2^n. & &
	\end{aligned}\right.
\end{equation}
 Analogous results to~\cref{lemma:rrpr} holds for the Dirichlet--Neumann and Neumann--Neumann methods. The following result gives that the three methods are well defined.
\begin{corollary}
    Suppose that~\cref{ass:phi,ass:eqdata,ass:tech} hold. The Dirichlet--Neumann, Neumann--Neumann, and Robin--Robin methods are well defined in the sense that each step of~\cref{eq:dn,eq:nn,eq:pr} has a unique solution.
\end{corollary}
\begin{proof}
    For the Dirichlet--Neumann and Neumann--Neumann methods the results follow immediately from~\cref{cor:spiso} since this implies that the interface iterations have unique iterates. For the Robin--Robin method the result also follows since
    \begin{equation*}
        \langle \mathcal{R}\mu, \mu\rangle\geq\|\mu\|_{L^2_{L^2(\Gamma)}(\R_+)}^2\geq 0 \quad\text{for all }\mu\in Z,
    \end{equation*}
    which means that $s_0\mathcal{R}+S_i$ is coercive and therefore an isomorphism for all $s_0>0$.
\end{proof}

In order to illustrate the applicability of the derived framework, we will next prove convergence of the Robin--Robin method. We begin by introducing a particular Gelfand triple.
\begin{lemma}\label{lem:Zdense}
    If~\cref{ass:phi} holds, then $Z$ is densely embedded into $L^2_{L^2(\Gamma)}(\R_+)$.
\end{lemma}
\begin{proof}
     Note first of all that $Z \hookrightarrow L^2_{L^2(\Gamma)}(\R_+)$ is clearly a continuous embedding. For density, let $\eta \in L^2_{L^2(\Gamma)}(\R_+)$ and define 
    \begin{equation*}
        \mu= \psi_-\eta\in L^2\bigl(\R_+; L^2(\Gamma(0))\bigr).
    \end{equation*}
    Now recall that $H^{1/4}(\R_+)$ is dense in $L^2(\R_+)$ and $\Lambda(0)$ is dense in $L^2(\Gamma(0))$, see~\cite[Lemma 4.2]{EHEE22}. Therefore the algebraic tensor product space $H^{1/4}(\R_+)\otimes \Lambda(0)$ is dense in $L^2\bigl(\R_+; L^2(\Gamma(0))\bigr)$, see~\cite[Theorem 3.12]{weidmann}. It then follows that the larger space
    \begin{equation*}
        H^{1/ 4}\bigl(\R_+; L^2(\Gamma(0))\bigr)\cap L^2\bigl(\R_+; \Lambda(0)\bigr)
    \end{equation*}
    is also dense in $L^2\bigl(\R_+; L^2(\Gamma(0))\bigr)$. Hence, there exists a sequence
    \begin{equation*}
        \mu_n \in H^{1/ 4}\bigl(\R_+; L^2(\Gamma(0))\bigr)\cap L^2\bigl(\R_+; \Lambda(0)\bigr)
    \end{equation*}
    that satisfies $\mu_n \to \mu$ in $L^2\bigl(\R_+; L^2(\Gamma(0))\bigr)$. It immediately follows from the fact that $\psi$ is an isomorphism that the sequence $\eta_n=\psi \mu_n$ satisfies $\eta_n\in Z$ and $\eta_n\to \eta$ in $L^2_{L^2(\Gamma)}(\R_+)$. Since $\eta$ was arbitrary this shows that $Z$ is dense in $L^2_{L^2(\Gamma)}(\R_+)$.
\end{proof}
By~\cref{lem:Zdense} together with the fact that $Z$ and $L^2_{L^2(\Gamma)}(\R_+)$ are Hilbert spaces, we have the Gelfand triple
    \begin{equation*}
        Z\hookrightarrow L^2_{L^2(\Gamma)}(\R_+)\hookrightarrow Z^*,
    \end{equation*}
with dense embeddings. We recall the Riesz isomorphism $\mathcal{R}\colon L^2_{L^2(\Gamma)}(\R_+)\rightarrow L^2_{L^2(\Gamma)}(\R_+)^*$:
\begin{equation}\label{eq:riesz}
\langle \mathcal{R}\eta, \mu\rangle=(\eta, \mu)_{L^2_{L^2(\Gamma)}(\R_+)}\quad\text{ for all } \eta\in L^2_{L^2(\Gamma)}(\R_+),\,\mu\in Z.
\end{equation}
We trivially have the bounds
\begin{equation}\label{eq:rieszbound}
\abs{\langle \mathcal{R}\eta, \mu\rangle}\leq\|\eta\|_{L^2_{L^2(\Gamma)}(\R_+)}\|\mu\|_{L^2_{L^2(\Gamma)}(\R_+)}\quad\text{ for all } \eta\in L^2_{L^2(\Gamma)}(\R_+),\,\mu\in Z
\end{equation}
and
\begin{equation}\label{eq:rieszcoer}
\langle \mathcal{R}\eta, \eta\rangle\geq\|\eta\|_{L^2_{L^2(\Gamma)}(\R_+)}^2\quad\text{ for all } \eta\in Z.
\end{equation}
The variational framework appears to be too general for analyzing the convergence of the Peaceman--Rachford iteration. Therefore, we introduce the Steklov--Poincaré operators as affine unbounded operators on $L^2_{L^2(\Gamma)}(\R_+)$ before we prove that the iteration converges. To this end, let 
\begin{align*}
        D(\mathcal{S}_i)&=\{\eta\in Z: S_i\eta-\chi_i\in L^2_{L^2(\Gamma)}(\R_+)^*\},\\
        D(\mathcal{S})&=\{\eta\in Z: S\eta-\chi\in L^2_{L^2(\Gamma)}(\R_+)^*\},
\end{align*}
and define the unbounded affine operators
\begin{align*}
	\mathcal{S}_i&:D(\mathcal{S}_i)\subseteq L^2_{L^2(\Gamma)}(\R_+)\rightarrow L^2_{L^2(\Gamma)}(\R_+): \eta\mapsto \mathcal{R}^{-1}(S_i\eta-\chi_i),\\
	\mathcal{S}&:D(\mathcal{S})\subseteq L^2_{L^2(\Gamma)}(\R_+)\rightarrow L^2_{L^2(\Gamma)}(\R_+): \eta\mapsto \mathcal{R}^{-1}(S\eta-\chi).
\end{align*}
In $L^2_{L^2(\Gamma)}(\R_+)$ the Steklov--Poincaré equation is to find $\eta\in D(\mathcal{S})$ such that
\begin{equation}\label{eq:spl2}
	  \mathcal{S}\eta=0
\end{equation}
and the Peaceman--Rachford iteration takes the following form: For each $n=1, 2,\dots$, find $(\eta_1^n, \eta_2^n)\in D(\mathcal{S}_1)\times D(\mathcal{S}_2)$ such that
\begin{equation}\label{eq:prl2}
  \left\{\begin{aligned} (s_0I+\mathcal{S}_1)\eta_1^n&=(s_0I-\mathcal{S}_2)\eta_2^{n-1},\\
        (s_0I+\mathcal{S}_2)\eta_2^n&=(s_0I-\mathcal{S}_1)\eta_1^n.
\end{aligned}\right.
\end{equation}
Here, $\eta_2^0\in D(\mathcal{S}_2)$ is a given initial guess. We now verify that~\cref{eq:spl2} is indeed a restriction of the weak Steklov--Poincaré equation~\cref{eq:speq}.
\begin{lemma}\label{lemma:spvl2}
Suppose that~\cref{ass:phi,ass:eqdata,ass:tech} hold. If $\eta\in D(\mathcal{S})$ solves the $L^2$-Steklov--Poincaré equation~\cref{eq:spl2}, then $\eta$ also solves the weak Steklov--Poincaré equation~\cref{eq:speq}.
\end{lemma}
\begin{proof}
From~\cref{eq:spl2} we have
\begin{displaymath}
(\mathcal{R}^{-1}(S\eta-\chi), \mu)_{L^2_{L^2(\Gamma)}(\R_+)}=0\quad\text{ for all }\mu\in L^2_{L^2(\Gamma)}(\R_+).
\end{displaymath}
Therefore, we get by~\cref{eq:riesz} that
\begin{displaymath}
\langle S\eta-\chi, \mu \rangle=(\mathcal{R}^{-1}(S\eta-\chi), \mu)_{L^2_{L^2(\Gamma)}(\R_+)}=0
\end{displaymath}
for all $\mu\in Z$.
\end{proof}
A similar result holds for the Peaceman--Rachford iteration. The proof is left out since it is similar to the proof of~\cref{lemma:spvl2}. 
\begin{lemma}\label{lemma:prl2weak}
Suppose that~\cref{ass:phi,ass:eqdata,ass:tech} hold. If $(\eta_1^n, \eta_2^n)_{n\geq 1}$ solves the $L^2$-Peaceman--Rachford iteration~\cref{eq:prl2} with $\eta_2^0\in D(\mathcal{S}_2)$, then $(\eta_1^n, \eta_2^n)_{n\geq 1}$ also solves the weak Peaceman--Rachford iteration~\cref{eq:pr} with the same initial guess.
\end{lemma}
\begin{lemma}
Suppose that~\cref{ass:phi,ass:eqdata,ass:tech} hold. Then $\mathcal{S}_i,\, i=1, 2$, satisfy the monotonicity property
\begin{equation}\label{eq:coerl2} 
(\mathcal{S}_i\eta-\mathcal{S}_i\mu, \eta-\mu)_{L^2_{L^2(\Gamma)}(\R_+)}\geq 
	c\|\eta-\mu\|^2_Z\quad \text{ for all }\eta, \mu\in D(\mathcal{S}_i).
\end{equation}
Moreover, for any $s_0\geq0$ the operators $s_0I+\mathcal{S}_i\colon D(\mathcal{S}_i)\rightarrow L^2_{L^2(\Gamma)}(\R_+),\, i=1, 2$, are isomorphisms. Similar results hold for $\mathcal{S}$. In particular, there exists a unique solution to \cref{eq:spl2} and the iteration \cref{eq:prl2} is well defined.
\end{lemma}
\begin{proof}
The monotonicity follows from~\cref{eq:riesz,eq:weakcoer}, since
\begin{align*}
      (\mathcal{S}_i\eta-\mathcal{S}_i\mu, \eta-\mu)_{L^2_{L^2(\Gamma)}(\R_+)}&=\langle (S_i\eta-\chi_i)-(S_i\mu-\chi_i), \eta-\mu\rangle\\
             &=\langle S_i(\eta-\mu), \eta-\mu\rangle\geq c\|\eta-\mu\|^2_Z
\end{align*}
for all $\eta, \mu\in D(\mathcal{S}_i)$. Let $\mu\in L^2_{L^2(\Gamma)}(\R_+)$ be arbitrary. Then $\chi_i+\mathcal{R}\mu\in Z^*$, and by~\cref{cor:spiso}, there exists a unique $\eta\in Z$ such that $(s_0\mathcal{R}+S_i)\eta=\chi_i+\mathcal{R}\mu$ in $Z^*$. Rearranging yields that $S_i\eta-\chi_i=\mathcal{R}(\mu-s_0\eta)\in L^2_{L^2(\Gamma)}(\R_+)^*$, i.e.,  $\eta\in D(\mathcal{S}_i)$ with
\begin{displaymath}
            (s_0I+\mathcal{S}_i)\eta=s_0\eta+\mathcal{R}^{-1}(S_i\eta-\chi_i)=\mu.
\end{displaymath}
Thus, we have shown that $(s_0I+\mathcal{S}_i)$ is an isomorphism. The proof for $\mathcal{S}$ is similar and is therefore left out.
\end{proof}

In order to proceed with the convergence analysis, we require the following regularity of the solution to the weak parabolic equation~\cref{eq:weakHonehalf}. 
\begin{assumption}\label{ass:regularity}
The functionals  
\begin{displaymath}
\mu\mapsto a_i(q_iu, F_i\mu)-\langle f_i, F_i\mu\rangle,\quad{i=1,2},
\end{displaymath}
are elements in $L^2_{L^2(\Gamma)}(\R_+)^*$, where $u\in W$ is the solution to~\cref{eq:weakHonehalf}.
\end{assumption}
This assumption is somewhat implicit, but can be interpreted as the solution $u$ having a normal derivative on $\Gamma$ in $L^2_{L^2(\Gamma)}(\R_+)$.
\begin{lemma}\label{lemma:friedrich}
Suppose that~\cref{ass:phi,ass:eqdata,ass:tech,ass:regularity} hold. If $\eta$ solves the $L^2$-Steklov--Poincaré equation~\cref{eq:spl2} then $\eta\in D(\mathcal{S}_1)\cap D(\mathcal{S}_2)$.
\end{lemma}
\begin{proof}
Let $\eta\in D(\mathcal{S})$ be the solution to~\cref{eq:spl2}. By~\cref{lemma:tpspeq,lemma:tranequiv}, we have the identification $q_iu=F_i\eta+G_if_i$, where $u\in W$ is the solution to the weak parabolic equation~\cref{eq:weakHonehalf}. \cref{ass:regularity} then yields that 
 \begin{align*}
            \mu\mapsto\langle S_i\eta-\chi_i, \mu\rangle&=a_i(F_i\eta, F_i\mu)+a_i(G_if_i, F_i\mu)-\langle f_i, F_i\mu\rangle\\
            &=a_i(q_iu, F_i\mu)-\langle f_i, F_i\mu\rangle
\end{align*}
is an element in $L^2_{L^2(\Gamma)}(\R_+)^*$, i.e., $\eta\in D(\mathcal{S}_i)$ for $i=1,2$.
\end{proof}
 
\begin{lemma}\label{lemma:prlimit}
Suppose that~\cref{ass:phi,ass:eqdata,ass:tech,ass:regularity} hold. Let $\eta$ be the solution to the $L^2$-Steklov--Poincaré equation~\cref{eq:spl2} and $(\eta_1^n, \eta_2^n)_{n\geq 1}$ be the $L^2$-Peaceman--Rachford iterates~\cref{eq:prl2} with $\eta_2^0\in D(\mathcal{S}_2)$. Then we have the limit
\begin{equation}\label{eq:prlimit}
 (\mathcal{S}_i\eta^n_i-\mathcal{S}_i \eta, \eta^n_i -\eta)_{L^2_{L^2(\Gamma)}(\R_+)}\rightarrow 0
\end{equation}
 as $n$ tends to infinity.
\end{lemma}
\cref{lemma:prlimit} is a consequence of the abstract result~\cite[Proposition 1]{lionsmercier}. The latter requires the monotonicity~\cref{eq:coerl2} and the fact that the solution to~\cref{eq:spl2} satisfies $\eta\in D(\mathcal{S}_1)\cap D(\mathcal{S}_2)$, which follows from~\cref{lemma:friedrich}. A simpler proof of~\cref{lemma:prlimit} can be found in~\cite[Lemma~8.8]{EHEE22}.

We are now in a position to prove that the Robin--Robin method converges.
\begin{theorem}\label{thm:rrconv}
Suppose that~\cref{ass:phi,ass:eqdata,ass:tech,ass:regularity} hold. Let $u$ be the solution to the weak parabolic equation~\cref{eq:weakHonehalf} and $\eta$ be the solution to the $L^2$-Steklov--Poincaré equation~\cref{eq:spl2}. The iterates $(\eta_1^n, \eta_2^n)_{n\geq 1}$ of the $L^2$-Peaceman--Rachford iteration~\cref{eq:prl2} converges to $\eta$, i.e.,
\begin{equation*}\label{eq:prconv}
\|\eta^n_1-\eta\|_Z+\|\eta^n_2-\eta\|_Z\rightarrow 0
\end{equation*}
as $n$ tends to infinity. Moreover the weak Robin--Robin approximation $(u_1^n, u_2^n)_{n\geq 1}$ converges to $(u_1, u_2)=(q_1u, q_2u)$, i.e.,
\begin{equation*}\label{eq:rrconv}
\|u^n_1- u_1\|_{W_1}+\|u^n_2-u_2\|_{W_2}\rightarrow 0
\end{equation*}
as $n$ tends to infinity.
\end{theorem}
\begin{proof}
From \cref{eq:prlimit,eq:coerl2} we obtain
\begin{align*}
            &\|\eta^n_1-\eta\|^2_Z+\|\eta^n_2-\eta\|^2_Z\\
            &\quad\leq C\bigl((\mathcal{S}_1\eta^{n+1}_1-\mathcal{S}_1 \eta, \eta^{n+1}_1 -\eta)_{L^2_{L^2(\Gamma)}(\R_+)}
                +(\mathcal{S}_2\eta^n_2- \mathcal{S}_2 \eta, \eta^n_2 -\eta)_{L^2_{L^2(\Gamma)}(\R_+)}\bigr)\rightarrow 0
\end{align*}
as $n$ tends to infinity. By~\cref{lemma:tpspeq,lemma:tranequiv,lemma:rrpr}, one has the identities
\begin{displaymath}
 (u_1, u_2)=(F_1\eta+G_1f_1, F_2\eta+G_2f_2)\quad\text{and}\quad (u_1^n, u_2^n)=(F_1\eta_1^n+G_1f_1, F_2\eta_2^n+G_2f_2).
\end{displaymath}
This together with the limit above and the fact that $F_i$ is bounded yields that 
\begin{align*}
            \|u^n_1-u_1\|_{W_1}+\|u^n_2-u_2\|_{W_2}&=\|F_1(\eta^n_1-\eta)\|_{W_1}+\|F_2(\eta^n_2-\eta)\|_{W_2}\\
            &\leq C\bigl(\|\eta^n_1-\eta\|_Z+\|\eta^n_2-\eta\|_Z\bigr)\rightarrow 0
\end{align*}
as $n$ tends to infinity.
\end{proof}

\begin{funding}
Engström and Hansen were supported by the Swedish Research Council under the grant 2023--04862. Djurdjevac acknowledges the support from DFG CRC/TRR 388 ``Rough
Analysis, Stochastic Dynamics and Related Fields'', Project B06.
\end{funding}

{\footnotesize
\bibliographystyle{plain}
\bibliography{refRRmovdom}
}

\end{document}